\documentclass[a4paper,11pt]{amsart}
\usepackage{amsmath}
\usepackage{amsthm}
\usepackage{amssymb}
\usepackage{amsfonts}

\usepackage{graphicx}
\usepackage{multicol}
\usepackage{fancyhdr}
\usepackage{tikz}
\usepackage{booktabs}
\usepackage{float}
\usepackage[toc,page]{appendix}
\usepackage{afterpage}
\usepackage[english]{babel}

\usepackage{appendix}

\usepackage{enumerate}
\usepackage[all]{xy}
\usepackage{graphicx}
\usepackage{multicol}
\usepackage{fancyhdr}
\usepackage{tikz}

\usepackage[utf8]{inputenc}

\usepackage[utf8]{inputenc}

\normalsize
\renewcommand{\Im}{\mbox{Im}}
\newtheoremstyle{theoremstyle}
  {10pt}      
  {5pt}       
  {\itshape}  
  {}          
  {\bfseries} 
  {:}         
  {.5em}      
  {}          

\newtheoremstyle{examplestyle}
  {10pt}      
  {5pt}       
  {}          
  {}          
  {\bfseries} 
  {:}         
  {.5em}      
  {}          

\theoremstyle{theoremstyle}
\newtheorem{theorem}{Theorem}[section]
\newtheorem*{theorem*}{Theorem}

\newtheorem*{theorema*}{Theorem A}

\newtheorem*{theoremb*}{Theorem B}

\newtheorem*{theoremc*}{Theorem C}

\newtheorem{lemma}[theorem]{Lemma}
\newtheorem*{lemma*}{Lemma}
\newtheorem{proposition}[theorem]{Proposition}
\newtheorem*{proposition*}{Proposition}
\newtheorem{corollary}[theorem]{Corollary}
\newtheorem*{corollary*}{Corollary}

\newtheorem*{counterexample*}{Counterexample}

\theoremstyle{examplestyle}

\newtheorem*{gp*}{Geometric Principle}

\newtheorem{definition}[theorem]{Definition}
\newtheorem{definition*}{Definition}
\newtheorem{remark}[theorem]{Remark}
\newtheorem{remark*}{Remark}

\newtheorem{convention*}{Convention}

\newtheorem{conjecture*}{Conjecture}

\newtheorem{notation*}{Notation}

\newtheorem{stl*}{Structure Lemma}
\newtheorem*{tha*}{Theorem A}
\newtheorem*{thb*}{Theorem B}
\newtheorem*{thc*}{Theorem C}

\newcommand{\comment}[1]{}

\newcommand{\rk}{\operatorname{rk}}

\newcommand{\Span}{\operatorname{span}}

\newcommand{\auxqed}[1][0.5]{\vskip -\belowdisplayskip\vskip-#1\baselineskip}

\newcommand{\eins}{{\mathchoice
{\mathrm 1\mskip-4.2mu\mathrm l}{\mathrm 1\mskip-4.2mu\mathrm l}
{\mathrm 1\mskip-3.9mu\mathrm l}{\mathrm 1\mskip-4.0mu\mathrm l}}}

\newcommand{\owedge}{\mathbin{\vcenter{\hbox{$\scriptstyle\bigcirc\mkern-14.7mu\wedge\mkern 2.7mu$}}}}
\newcommand{\N}{\mathbb{N}}
\newcommand{\Z}{\mathbb{Z}}
\newcommand{\R}{\mathbb{R}}
\newcommand{\C}{\mathbb{C}\mkern1mu}
\newcommand{\F}{\mathbb{F}\mkern1mu}
\renewcommand{\H}{\mathbb{H}\mkern1mu}
\newcommand{\Ca}{\mathbb{O}}

\newcommand{\FP}{\mathbb{F\mkern1mu P}}
\newcommand{\CP}{\mathbb{C\mkern1mu P}}
\newcommand{\Sph}{\mathbb{S}}

\DeclareMathOperator{\tr}{trace}

\newcommand{\bmat}{\left(\begin{smallmatrix}}
\newcommand{\emat}{\end{smallmatrix}\right)}

\newcommand{\SO}{\mathrm{SO}}

\newcommand{\SU}{\mathrm{SU}}

\newcommand{\Gtwo}{\mathrm{G}_2}

\begin{document}
\title{A new structural approach to isoparametric hypersurfaces in spheres}
\title[A new structural approach to isoparametric hypersurfaces in spheres]{A new structural approach to isoparametric hypersurfaces in spheres}
\author{Anna Siffert$^1$}
\address{Anna Siffert, Max Planck Institute for Mathematics, Vivatsgasse 7, 53111 Bonn, Germany}
\email{siffert@mpim-bonn.mpg.de}
\thanks{${}^1$ This work was finished while the author was supported by DFG grant SI 2077/1-1.}

\begin{abstract}
In this paper we show that the longstanding problem of classifying all isoparametric hypersurfaces in spheres with six different principal curvatures is still not complete. 
Moreover, we develop a structural approach that may be helpful for such a classification.
Instead of working with the isoparametric hypersurface family in the sphere, we consider the associated Lagrangian submanifold of the real Grassmannian of oriented $2$-planes in $\R^{n+2}$. We obtain new geometric insights into classical invariants and identities in terms of the geometry of the Lagrangian submanifold.
\end{abstract}

\maketitle

\section*{Introduction}
Originally, isoparametric hypersurfaces were defined to be the level sets of isoparametric functions, i.e., functions on a real space form whose gradient norm and Laplacian are constant on the level sets. 
This condition translates into the equivalent geometric condition that the principal curvatures of the hypersurfaces are constant. 
 The classifications of isoparametric hypersurfaces in the cases where the ambient space is Euclidean 
or hyperbolic space were settled soon by Somigliana \cite{so}, Segre \cite{se}, and Cartan \cite{cartan1}-\cite{cartan4}. 
In contrast, when the ambient space is a sphere, the number $g$ of distinct principal curvatures can 
be greater than two, which makes a classification much more difficult.
In this paper we henceforth consider the case where the ambient space is a sphere.

\smallskip

Cartan \cite{cartan1}-\cite{cartan4} 
classified isoparametric hypersurfaces with $g\leq 3$, and showed that they are all homogeneous, i.e. orbits of isometric group actions on $\Sph^{n+1}$.
The problem was picked up again by M\"unzner \cite{munzner, munzner2} who showed that the number of distinct principal
 curvatures $g$ can be only $1$, $2$, $3$, $4$, or $6$, and gave restrictions for the
 multiplicities as well. 
 The possible multiplicities of the curvature distributions were classified in \cite{munzner2}, \cite{abresch}, \cite{stolz},
 and coincide with the multiplicities in the known examples.
The situation in the case $g = 4$ is more complex since there exist infinitely many isoparametric hypersurfaces
 and infinitely many of them are inhomogeneous --- see Cecil \cite{cecil} or Thorbergsson \cite{tb2} for very good recent surveys of this case. In the case $g=6$ all multiplicities $m$ coincide and equal either $1$ or $2$, and precisely two homogeneous examples are known. Dorfmeister and Neher \cite{dn} conjectured that all isoparametric hypersurfaces with $g=6$ are homogeneous and
 in the same paper settled this conjecture in affirmative for $m=1$. 
 In the remaining open case $m = 2$, Miyaoka in \cite{mi2}, \cite{mierr} proposed how to establish homogeneity. 
 However, in the appendix we give a counterexample to one crucial proposition in \cite{mi2}, \cite{mierr}.
 Thus the classification of isoparametric hypersurfaces with $g=6$ and $m=2$ is still open.
 
\smallskip

In the present paper we develop a new structural approach to isoparametric hypersurfaces in spheres, unifying many of the known geometric properties. 

\smallskip

The basic idea is as follows. Instead of working 
with the family of parallel surfaces $F_t:M^n\rightarrow \Sph^{n+1}$, with normal 
field $\nu_t\in\Gamma(\nu M^n)$, one considers the associated submanifold of the real Grassmannian of oriented $2$-planes in $\R^{n+2}$
$$\frak{L}: M^n\rightarrow Gr^{+}_2(\R^{n+2})\subset \CP^{n+1}$$ 
by sending $p\in M^n$ to the $2$-plane spanned by $F_t(p)$ and $\nu_t(p)$.
One easily sees that $\frak{L}$ does not depend on $t$, and in \cite{palm1, palm2}
it was observed that for any submanifold of $\Sph^{n+1}$ the associated submanifold $\frak{L}(M)\subset Gr^{+}_2(\R^{n+2})$ 
is Lagrangian with respect to the natural symplectic structure.

\smallskip

We endow the Lagrangian submanifold with a set of invariants 
which arise naturally:
the metric induced via the canonical K\"ahler metric $g_{Q}$ on $Gr^{+}_2(\R^{n+2})$, i.e., $\hat{g}=\frak{L}^{*}g_{Q}$; the symmetric
tensor $\alpha$ $$\alpha(X,Y,Z)=g_t((\nabla_X^tA_t)Y,Z),$$ where $A_t$ denotes the shape operator of $F_t$ with
 respect to $\nu_t$, $X,Y,Z\in\Gamma(TM)$; and $$B\otimes B^{-1}:=B_t\otimes B_t^{-1},$$
 where $B_{t}:\Gamma(TM)\rightarrow \Gamma(TM)\otimes\C$ is defined via $B_{t}=(A_{t}+i\eins)\,(A_{t}-i\eins)^{-1}.$
The set of invariants $(\hat{g},\alpha,B\otimes B^{-1})$ depends only on the isoparametric family it is contained in.

\smallskip

The tensor $\alpha$ is one of the fundamental invariants used in the previous classification approaches though
usually encoded in some much less invariant Maurer-Cartan forms $\Lambda$.
The really interesting fact is that $\alpha$ coincides - up to a factor of two - with the 
second fundamental form of the Lagrangian submanifold, which gives a new, 
geometric interpretation for this important tensor.

 \begin{theorema*}
The tensor $\alpha$ coincides, up to a factor of two, with the second fundamental form of the Lagrangian submanifold. 
\end{theorema*}

The introduction of the set of invariants $(\hat{g},\alpha,B\otimes B^{-1})$ is justified by the fact 
that they allow us to formulate all relevant identities in a compact way.

\smallskip

For the theory of isoparametric hypersurfaces in spheres, the so-called Weyl Identities are of utmost importance.
The classical Weyl Identities depend on several indices. 
In terms of the invariants described above, these multiple identities unify into one tensor identity, see Theorem\,\ref{weylt}.

\smallskip

Another set of equations which can easily be formulated in terms of $\alpha$ are the Symmetry identities:
the pullback of $\alpha$ under the reflections through each of the focal manifolds coincides with the negative of $\alpha$.
 So far, all these considerations are completely general.

\smallskip

Below we restrict ourselves to the case $g=6$, and give several reformulations for homogeneity.
Here, we just mention the following reformulation, which  relates homogeneity to geometric properties of the Lagrangian submanifold.
For other reformulations see Section \ref{sec4}.

\begin{theoremb*}
The homogeneity of isoparametric hypersurfaces with $g=6$ different principal curvatures is equivalent to
\begin{align*}
R(\pi_iX,\pi_{i+3}Y,\pi_{i+3}Y,\pi_iX)=0,\,\forall i\in\left\{1,...,6\right\}, \,\forall X,Y\in\Gamma(TM)
\end{align*}
where $R$ is the curvature tensor of the Lagrangian submanifold $\frak{L}(M^n)$.
\end{theoremb*} 

The previous theorem thus reduces the classification of isoparametric hypersurfaces with $g=6$ to a geometric problem for Lagrangian submanifolds of the complex quadric.
We hope that our approach might finally lead to a classification of isoparametric hypersurfaces with $g=6$.
Note that up till now, for the classification of isoparametric hypersurfaces with $g=6$, there exists only one approach suggested by Dorfmeister and Neher \cite{dn}.
However, for the remaining case $m=2$, the underlying algebraic problem turned out to be very hard, and no one was so far able to carry out their approach for $m=2$.
Therefore a new approach seems to be needed.

\smallskip

\thispagestyle{empty}

Using this structural approach we can reprove many of the classical results. 
Most proofs become simpler and render greater geometric insight.

\smallskip

This paper is organized as follows: in Section\,\ref{sec1} we recall a few preliminary definitions and give 
a survey of results needed later on.
In Section \,\ref{sec2} we carry out the translation from the isoparametric hypersurface family in the sphere to the Lagrangian submanifold of the real Grassmannian of oriented $2$-planes in $\R^{n+2}$.
In particular, we introduce the set $(\hat{g},\alpha,B\otimes B^{-1})$ of structural invariants.
Section\,\ref{sec3} deals with the fundamental submanifold equations of the Lagrangian submanifold.
Moreover, we derive the Weyl Identities and the Symmetry identities.
In Section\,\ref{sec4} we give several equivalent formulations of homogeneity.
Finally, we provide a counterexample to one crucial proposition in Miyaoka's approach \cite{mi2}, \cite{mierr} in the appendix.

\section*{Acknowledegements}
It is a pleasure to thank Prof. Dr. Uwe Abresch, for introducing me to isoparametric hypersurfaces
and for his generous support. I benefitted enormously from numerous mathematical discussions with him.
Furthermore, I am very grateful to Prof. Dr. Wolfgang Ziller and Prof. Dr. Gudlaugur Thorbergsson for their many valuable comments and for fruitful mathematical discussions.
Moreover, I would like to thank the referee for carefully reading the manuscript and for giving such constructive comments
which substantially helped to improve the quality of the manuscript.

\smallskip

This work was finished while the author was supported by Deutsche Forschungsgemeinschaft with the grant SI 2077/1-1.
Last, but not least, I would like to thank the Max Planck Institute for Mathematics in Bonn for providing excellent working conditions.

\section{Preliminaries}
\label{sec1}
In this section we gather the definitions and tools which we will need later on.
Throughout this paper, let $M$ be a connected, smooth manifold of dimension $n$.

\smallskip

\subsection*{Isoparametric hypersurface}
We start by recalling the definition of an isoparametric hypersurface in a sphere.

\begin{definition}
An embedding $F_{0}:{M}\hookrightarrow \Sph^{n+1}$ together 
with a distinguished unit normal vector field $\nu_0\in\Gamma(\nu M)$ is called an {\em isoparametric hypersurface in $\Sph^{n+1}$}
if and only if the principal curvature functions are constant.
\end{definition}

The data of an isoparametric hypersurface are $g$, the number of distinct principal curvatures, and  $m_j$ ($1\leq j\leq g$), the multiplicities of the curvature distributions.

\smallskip

Below we assume that $(F_0,\nu_0)$ constitutes an isoparametric hypersurface in $\Sph^{n+1}$.
Let $A_0$ be the shape operator of $F_0$ with respect to $\nu_0$.
We denote the $g$ different eigenvalues of $A_0$, i.e. the principal curvature functions,
by $\lambda_{j}(0)$. By assumption the $\lambda_j(0)$ are constant on $M$.
Without loss of generality, we assume $\lambda_1(0)>...>\lambda_g(0)$ and
define $\theta_j\in\left(-\frac{\pi}{2},\frac{\pi}{2}\right)$ such
that $\lambda_j(0)=\cot(\theta_j)$.

\smallskip

We denote the curvature distribution associated to $\lambda_j(0)$ by $D_j$, i.e. we have $D_j=\mbox{Eig}(A_0,\lambda_j(0))$.
Furthermore, we let $\pi_j:TM\rightarrow D_j$ be the orthogonal projection into $D_j$.
Note that we have $m_j=\dim D_j$.
The following lemma is due to M\"unzner \cite{munzner}.

\begin{lemma}[\cite{munzner}]
The curvature distribution $D_j$ is integrable and the leaves are small spheres in $\Sph^{n+1}$ with curvature $1+\cot^2(\theta_j)$.
These spheres are totally geodesic submanifolds of $M$.
\end{lemma}

\subsection*{Parallel surfaces}
Isoparametric hypersurfaces always come along as families of isoparametric hypersurfaces; namely \lq almost all\rq\ parallel surfaces of a given isoparametric hypersurface are also isoparametric hypersurfaces.

\smallskip

In what follows let  $F_{0}:{M}\hookrightarrow \Sph^{n+1}$ together with a distinguished normal vector field $\nu_0$ be a fixed isoparametric hypersurface. By slight abuse of notation
we also call the image $F_0(M)$ an isoparametric hypersurface.
We consider the \textit{parallel surface} $F_{t}: M\hookrightarrow\Sph^{n+1}$ with signed distance $t$ to $F_0$.
It is given by
\begin{align*}
p\mapsto F_t(p):=\exp_{F_0(p)}(t\nu_{0\vert p})=\cos(t)F_0(p)+\sin(t)\nu_{0\vert p},
\end{align*}
endowed with the orientation
\begin{align*}
\nu_{t}(p)=-\sin(t)F_{0}(p)+\cos(t){\nu_{0}}_{\vert_p}.
\end{align*} 

The map $F_t$ induces the following data on $M$: the Riemannian metric 
$g_t=F_t^{*}\left\langle \cdot,\cdot\right\rangle_{\Sph^{n+1}}$, the associated Levi-Civita connection $\nabla^t$ and
the shape operator $A_t$ of the submanifold $(M,g_{t})\subset(\Sph^{n+1},\left\langle \cdot,\cdot\right\rangle_{\Sph^{n+1}})$ 
with respect to $\nu_t$.

\smallskip

In the next lemma we express $A_t$ in terms of $A_0$ since this will be needed later on.         

\begin{lemma}
\label{shapeoperator}
In terms of $A_0$ the shape operator $A_t$ is given by
\begin{align*}
{A_{t}}=(\eins+\cot(t){A_{0}})(\cot(t)\eins-{A_{0}})^{-1},
\end{align*}
where in the cases $t=\theta_{j}+\ell\pi$, $\ell\in\Z_2$, the operator $A_{t}$ is defined on $TM\!\setminus\! D_j$. 
\end{lemma}
\begin{proof}
Since $$d\nu_t=-dF_0(\sin(t)\eins+\cos(t)A_0),$$
equation (\ref{abf}) implies
\begin{align}
\label{abn}
d{\nu_{t}}=-{dF_{t}}(\eins+\cot(t){A_{0}})(\cot(t)\eins-{A_{0}})^{-1},
\end{align}
whence the claim.
\auxqed
\end{proof}

We will now make sense of the statement that  \lq almost all\rq\ parallel surfaces of a given isoparametric hypersurface are also isoparametric hypersurfaces.
Using the identity $d \nu_{0}=-dF_{0}A_{0}$ we get
\begin{align}
{dF_{t}}X={dF_{0}}(\cos(t)\eins- \sin(t){A_{0}})X
\label{abf}
\end{align}
and hence $\rk (dF_{t\vert p})=n$ if $t\neq\theta_j\,\,\mbox{mod}\,\pi$ and
$\rk (dF_{t\vert p})=n-m_j$ otherwise.

\smallskip

If $t\neq\theta_j\,\,\mbox{mod}\,\pi$, the parallel surface $F_t(M)$ is again an isoparametric hypersurface.
Lemma \ref{shapeoperator} implies that the principal curvatures of $F_t(M)$ are given by $\lambda_j(t)=\cot(\theta_j-t)$.

\smallskip

If $t=\theta_j\,\,\mbox{mod}\,\pi$, the $m_j$-dimensional eigenspace $D_j(p)$ is the kernel of $dF_{t\vert p}$ for every $p\in M$.
 In other words, the leaf $L_j(p)$ of $D_j$ through $p$, is focalized into the point $\overline{p}=F_{\theta_j}(p)$.
Hence,  $M_{j}:=F_{\theta_j}(M)$ is a so-called focal submanifold of dimension 
$(n-m_j)$.

\smallskip

Summarizing the previous considerations we get
$$F_0\,\,\mbox{isoparametric}\Rightarrow F_t \,\,\,\mbox{is}\,\,\left\{\begin{array}{ll} \mbox{isoparametric},& \mbox{if $\cot(t)\notin\mbox{spec}(A_0)$;} \\
         \mbox{submersion onto a focal manifold},&\mbox{otherwise.}
         \end{array}\right.$$

Finally, it is important  to remark that all the hypersurfaces in a family of parallel isoparametric hypersurfaces have the same focal submanifolds.
It is easily shown, see e.g. \cite{munzner}, that there are exactly two focal submanifolds.

\subsection*{Spectrum of the focal shape operators}
Using identity (\ref{abf}),
M\"unzner \cite{munzner} proved that for $t=\theta_j\,\,\mbox{mod}\,\pi$ the spectrum of $A_{t\vert \nu_p}$
is independent of $\nu\in\nu M_{j}$ and $p\in M_{j}$
and is given by  
\begin{align}
\label{specfoc}
\mbox{spec}(A_{t\vert\nu_p})=\left\{\cot\big((i-j)\pi/g \big)\, \vert\, i\in \left\{1,...,g\right\},\, i\neq j\right\}.
\end{align}
Thus for each $p\in M_{j}$ and each pair of orthonormal vectors $v_1,v_2\in\nu_p M_{j}$
the family
\begin{align*}
L(s)=\cos(s)A_{t\vert v_1}+\sin(s)A_{t\vert v_2},\,s\in\R,
\end{align*}
is isospectral. We will henceforth refer to $L(s)$ as the isospectral family at $p\in M_{j,\ell}$ with respect to $(v_1,v_2)\in\nu_p M_{j}$.

\smallskip

 The fact that the spectrum of the focal shape operator of the focal submanifold is independent of $\nu\in\nu M_{j}$ and $p\in M_{j}$ implies that the eigenvalues
$\lambda_k(0)$, $k\in\left\{1,...,g\right\}$,
are of the form $\lambda_k(0)=\cot(\phi+(k-1)\pi/g),$
with $0<\phi<\pi/g $. 
Thus $\theta_k=\phi+(k-1)\pi/g$ modulo $\pi$.

\smallskip

The parameter $\phi$ in the formula for $\theta_j$ encodes the position of $F_{0}$ in the isoparametric family.
We shall choose the initial hypersurface such that $\phi =\theta_1=\frac{\pi}{2g}$. Thus the initial isoparametric hypersurface is the one which lies in the middle of the focal submanifolds $F_{-\pi/2g}(M)$ and $F_{\pi/2g}(M)$. 

\smallskip

Using that the spectrum of $A_{t \nu_p}$
is independent of $\nu\in\nu M_{j}$, M\"unzner proved that
the multiplicities satisfy the equation 
\begin{align*}
m_{i}=m_{i+2}, i\in\Z_g.
\end{align*}
Therefore at most two distinct values for the multiplicities exist. They will henceforth be referred to as $m_1$ and $m_2$.
If all multiplicities coincide, what is for example the case if $g$ is odd,
their common value is denoted by $m$.

\subsection*{Global structure}
The global situation is as follows:
\begin{align*}
\frak{F}=\left\{F_t(M)\,\lvert\,t\in\left[-\pi/2g,\pi/2g\right]\right\}
\end{align*}
is a singular Riemannian foliation, $F_t(M)$ are isoparametric hypersurfaces for all $t\in\left(-\pi/2g,\pi/2g\right)$,
$F_{-\pi/2g}(M)$ and $F_{\pi/2g}(M)$ are submanifolds of codimension at least two in $\Sph^{n+1}$. 
Each normal geodesic $\gamma$ intersects the focal submanifolds at times
$t=(2j+1)\pi/2g$, $j\in\Z$, alternating between the two focal submanifolds $M_+:=F_{-\pi/2g}(M)$ and $M_-:=F_{\pi/2g}(M)$.
The regular set $\mathcal{R}$ is the set of times $t\in\R$ 
such that $\gamma(t)$ is not a focal point,
$$\mathcal{R}=\left\{t\in\R\,\lvert\,\not\exists\,j\in\Z\,\,\mbox{with}\,\,t=(2j+1)\pi/2g\right\}.$$ 
Any fixed isoparametric hypersurface $F_{t_0}(M)\in\mathfrak{F}$
coincides with either of the tubes $\mbox{Tube}_{d_+}(M_+)$
and $\mbox{Tube}_{d_-}(M_-)$ of radius $d_+$ and $d_-$, respectively, where $d_{\pm}$ denotes the distance of $M_{\pm}$ to $F_{t_0}(M)$.
Thus each normal geodesic intersects a given isoparametric hypersurface $F_{t_0}(M)$ exactly $2g$ times before it closes.
Furthermore, the focal set of each isoparametric hypersurface $F_{t_0}(M)$ is exactly the union of $M_+$ and $M_-$.
 
\smallskip

Figure\,\ref{fig1} sketches a normal geodesic in the case $g=3$.
 It intersects each isoparametric hypersurface exactly six times before it closes.
The intersection points with one fixed isoparametric hypersurface are marked by solid points.

\begin{figure}[h]
\begin{center}
\begin{tikzpicture}
\draw (0,0) circle [radius=1.5cm]; 
\draw[very thick](0:1.3cm)--(0:1.7cm);
\node at (0:2.1cm) {$M_+$};
\fill(10:1.5cm) circle[radius=0.05cm];
\fill(350:1.5cm) circle[radius=0.05cm];
\draw[very thick](60:1.3cm)--(60:1.7cm);
\node at (60:1.9cm) {$M_-$};
\draw[very thick](120:1.3cm)--(120:1.7cm);
\node at (120:2.0cm) {$M_+$};
\fill(130:1.5cm) circle[radius=0.05cm];
\fill(110:1.5cm) circle[radius=0.05cm];
\draw[very thick](180:1.3cm)--(180:1.7cm);
\node at (180:2.1cm) {$M_-$};
\draw[very thick](240:1.3cm)--(240:1.7cm);
\node at (240:2.0cm) {$M_+$};
\fill(230:1.5cm) circle[radius=0.05cm];
\fill(250:1.5cm) circle[radius=0.05cm];
\draw[very thick](300:1.3cm)--(300:1.7cm);
\node at (300:2.0cm) {$M_-$};
\end{tikzpicture}
\caption{Global picture for $g=3$}
\label{fig1}
\end{center}
\end{figure}
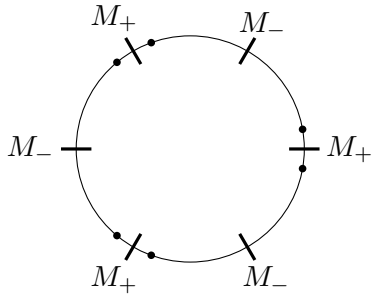

\subsection*{Topology}
\label{topo}
Each isoparametric hypersurface $F_{t_0}(M)$ separates the sphere $\Sph^{n+1}$ into two connected components $B_+$ and $B_-$, i.e. $F_{t_0}(M)=B_+\cap B_-$ and $\Sph^{n+1}=B_+\cup B_-$, such that these components are disk bundles over the focal manifolds. Assume without loss of generality that $M_+$ has codimension $m_1+1$ and that $M_-$ has codimension $m_2+1$. 
Thus we have the disk bundles 
\begin{align*}
D_{\pm}\rightarrow B_{\pm}\rightarrow M_{\pm},
\end{align*}
where the fibers $D_+$  and $D_-$ have dimensions $m_1+1$ and $m_2+1$, respectively.

\smallskip

This topological fact was used in the papers \cite{munzner2}, \cite{abresch} and \cite{stolz} to classify the number
of distinct principal curvatures and their possible multiplicities.
 
\subsection*{Classification results}
In Table\,1 the known classification results for isoparametric hypersurfaces in spheres
 with $g$ different principal curvatures are summarized.
 Recall that if the multiplicities $m_1$ and $m_2$ coincide their common value is denoted by $m$. 
 
\renewcommand{\thefootnote}{\alph{footnote}}
\begin{table}[h]
\begin{center}
\caption{Classification results of isoparametric hypersurfaces in spheres}
\label{tabelle1}
\begin{tabular}{l|l}
\toprule
$g$ &ISOPARAMETRIC HYPERSURFACES IN $\Sph^{n+1}$\\  \midrule
1& Open subset of a great or small hypersphere in $\Sph^{n+1}$.\\[0.3em]
2& Standard product of two spheres $\Sph^{d_1}(r_1)\times\Sph^{d_2}(r_2)\subset\Sph^{n+1}$ \\
&with $r_1^2+r_2^2=1$ and $n=d_1+d_2$.\\ [0.3em]
3& $m\in\left\{1,2,4,8\right\}$; tube of constant radius of a standard Veronese embedding\\
& of a projective plane $\FP^2$ into $\Sph^{3m+1}$, where $\F$ equals the\\
& division algebra $\R$, $\C$, $\H$ or $\Ca$,
for $m=1,2,4,8$, respectively.\\ [0.3em]
4& Multiplicities coincide with those of the examples of FKM-type\\
& or the two homogeneous exceptions $(m_1,m_2)=(2,2)$ and $(4,5)$.\\
&If $(m_1,m_2)=(2,2)$ or $(4,5)$ they are homogeneous.\\  
&If $m_2\geq 2m_1-1$ or if $(m_1,m_2)=(3,4)$, $(6,9)$ or $(7,8)$ they are of FKM-type.\\[0.3em]
6& $m\in\left\{1,2\right\}$.  If $m=1$, the hypersurface is homogeneous.\\
& For the remaining case $m=2$ the classification is still open.\\ \bottomrule
\end{tabular}\\
\end{center}
\end{table}

\smallskip

The classification of isoparametric hypersurfaces with $g\leq 3$ is due to Cartan \cite{cartan1}-\cite{cartan4}.

\smallskip

While all isoparametric hypersurfaces with $g\leq 3$ are homogeneous, the situation in the case $g =4$ is more complex since there exist infinitely many isoparametric hypersurfaces and infinitely many of them are inhomogeneous.\\
Ferus, Karcher and M\"unzner \cite{fkm} used representations of Clifford algebras to produce a class of isoparametric families with four principal curvatures, the so-called isoparametric hypersurfaces of FKM-type.
Stolz \cite{stolz} proved that the multiplicities of an isoparametric hypersurface with four principal curvatures must coincide with those in the known examples of FKM-type or two homogeneous exceptions, namely $(m_1,m_2)=(2,2)$ or $(4,5)$. 
Cecil, Chi and Jensen \cite{ccj} proved that if the multiplicities $(m_1,m_2)$ satisfy $m_2\geq 2m_1-1$, then the isoparametric hypersurfaces is necessarily of FKM-type.
Thus the cases $(m_1,m_2)=(4,5), (3,4), (6,9)$ and $(7,8)$ were left, which were successively classified by Chi, see \cite{chi} and the references therein.
For a detailed exposition to the the cases $g=4$ we refer the reader to the surveys by Cecil \cite{cecil} and Thorbergsson \cite{tb2}.

\smallskip

We postpone a detailed exposition of the case $g=6$ to Section \ref{sec4}, since in order to do so, we make use of results presented in Section \ref{sec2} and \ref{sec3}.

\section{The Lagrangian submanifold model and its structural invariants}
\label{sec2}
As already discussed in Section\,1, isoparametric hypersurfaces always come along as families of isoparametric hypersurfaces.
To each family of isoparametric hypersurfaces, we associate a Lagrangian submanifold of the complex quadric.
Thus, instead of working with a family of hypersurfaces, we can henceforth work with only \textit{one} submanifold. 
For this submanifold we introduce a set of invariants $(\hat{g},\alpha,B\otimes B^{-1})$ and endow them with a geometric meaning.

\smallskip

This section is structured as follows: in the first subsection we recall the definition and basic properties of the complex quadric.
The construction of the Lagrangian submanifold of the complex quadric is explained in the second subsection.
Finally, we introduce the set of invariants in the third subsection.
Throughout this section let $X,Y,Z,W\in\Gamma(TM)$.

\subsection{The complex quadric}
In this subsection we give a very brief introduction to the complex quadric.
We just cover those definitions and facts needed later on.
For a detailed exposition we refer the reader to the book \cite{gg} of Gasqui and Goldschmidt which we use as reference.

\medskip

We write $\left\langle \cdot,\cdot\right\rangle_{\C^{n+2}}$  and $\left\langle \cdot,\cdot\right\rangle_{h}$ for the standard complex bilinear and
 the standard hermitian inner product of $\C^{n+2}$, respectively. 
 
 \begin{definition}
The complex quadric is the complex hypersurface of $\CP^{n+1}$ given by
\begin{align*}
Q_n=\left\{\left[z\right]\in\CP^{n+1}\,\lvert\,z_0^2+...+z_{n+1}^2=0\right\},
\end{align*}
where $z=(z_0,...,z_{n+1})$ denote the standard coordinates of $\C^{n+2}$. 
\end{definition}

Clearly, the complex quadric $Q_n$ may also be described by
\begin{align*}
Q_n=\left\{\pi(z)\vert \, z\in\Sph^{2n+3}, \left\langle z,z\right\rangle_{\C^{n+2}}=0\right\},
\end{align*}
 where $\pi:\C^{n+2}\setminus\left\{0\right\}\rightarrow \CP^{n+1}$ is the natural projection, where $\Sph^{2n+3}\subset\C^{n+2}$ is endowed with the Riemannian metric induced by $\langle\cdot,\cdot\rangle_{\R}=\mbox{Re}\langle\cdot,\cdot\rangle$ on $\C^{n+2}$.
 
 \smallskip

Another well-known fact is that the complex quadric $Q_n$ is diffeomorphic to the real Grassmannian $Gr^{+}_2(\R^{n+2})$ of oriented $2$-planes in $\R^{n+2}$.
From now on we shall use this identification whenever convenient. 
 
 \medskip

 The remaining part of this subsection aims at describing the curvature tensor of $Q$. 
 In order to do so, we first have to introduce some notation.
 
 \smallskip
 
  Let $g_Q$ denote the K\"ahler metric on $Q_n$ induced from the Fubini-Study metric $g_{FS}$ on $\CP^{n+1}$ (of constant holomorphic curvature $4$) by the inclusion map $\iota:Q_n\rightarrow \CP^{n+1}$, i.e.,
$g_Q=\iota^{*}g_{FS}$.
The associated Levi-Civita connection of $Q_n$ is denoted by $\nabla^Q$.
For both, $\CP^{n+1}$ and $Q_n$, the complex structure  
shall be called $J$ and the associated K\"ahler form $\omega$.

\smallskip

It  is well-known  that the projection $\pi:\Sph^{2n+3}\rightarrow\CP^{n+1}$ is a Riemannian submersion and that
the map $d\pi:H_z\rightarrow T_{\pi(z)}\CP^{n+1}$
is an isometry, where
\begin{align*}
H_z=\left\{u\in\C^{n+2}\vert\, \left\langle z,u\right\rangle_{\C^{n+2}}=0\right\}.
\end{align*}
For a point $z\in\Sph^{2n+3}$ which satisfies $\left\langle z,z\right\rangle_{\C^{n+2}}=0$ 
we get an isometry
$d\pi:H_z^{'}\rightarrow T_{\pi(z)}Q$, where $H_z^{'}$ is the subspace of $H_z$ defined by
\begin{align}
\label{hz}
H_z^{'}=\left\{u\in\C^{n+2}\vert\, \left\langle z,u\right\rangle_{\C^{n+2}}=0,\, \left\langle z,u\right\rangle_h=0\right\}.
\end{align}
The preceding considerations allow us to identify the tangent space $T_qQ^n$ of the complex quadric at a given point $q\in Q^n$ with $\C^n$.
Then the complex structure on $T_qQ^n$ is given by multiplication by $i$ on $\C^n$, and the K\"ahler metric $g_Q$ corresponds to the real part of the standard Hermitian inner product on $\C^n$. 
Below we shall use this identification whenever convenient. 

\smallskip

Gathering the preceding information and carrying out a straightforward calculation, one obtains the following lemma.

\begin{lemma}[\cite{gg}]
\label{hrq}
The Riemann curvature tensor of the quadric $Q^n$ is given by
\begin{align*}
R^{Q^n}=g_Q\owedge g_Q+\omega\owedge\omega+q\owedge\overline{q}
\end{align*}
where $q(\,\cdot\,,\,\cdot\,)=\langle \,\cdot\,,\,\cdot\,\rangle_{\C^{n+2}}$ and $\overline{q}(\,\cdot\,,\,\cdot\,)=\overline{\langle \,\cdot\,,\,\cdot\,\rangle}_{\C^{n+2}}$.
\end{lemma}

For the convenience of the reader we recall the definition of the Kulkarni-Nomizu product. 

\begin{definition}
The Kulkarni-Nomizu product 
\begin{enumerate}
\renewcommand{\labelenumi}{(\roman{enumi})}
\item of two symmetric $(2,0)$-tensors $h_1$ and $h_2$ is the $(4,0)$-tensor $h_1\owedge h_2$ given by
\begin{align*}
h_1\owedge h_2(X,Y,Z,W)=\frac{1}{2}\big(&h_1(X,W)h_2(Y,Z)+h_2(X,W)h_1(Y,Z)\\-&h_1(X,Z)h_2(Y,W)-h_2(X,Z)h_1(Y,W)\big).
\end{align*}
\item  of the skew-symmetric form $\omega$ with itself is given by
\begin{align*}
(\omega \owedge\omega)(X,Y,Z,W)=\omega(X,W)&\omega(Y,Z)-\omega(X,Z)\omega(Y,W)\\&-2\omega(X,Y)\omega(Z,W).
\end{align*}
\end{enumerate}
\end{definition}

\subsection{From families of isoparametric hypersurfaces to a Lagrangian submanifold of the complex quadric}
In \cite{palm1, palm2} Palmer showed that every oriented, immersed hypersurface in the sphere naturally leads to a Lagrangian submanifold of the complex quadric.
We will apply this construction to isoparametric hypersurfaces in spheres.
In particular, we prove that every isoparametric hypersurface of a given family of isoparametric hypersurfaces leads to the same Lagrangian submanifold of the sphere.

\smallskip

We start by recalling the definition of the Stiefel manifold.
\begin{definition}
The Stiefel manifold is given by
\begin{align*}
St_2(\mathbb{R}^{n+2})=\left\{(v,w)\in\Sph^{n+1}\times\Sph^{n+1}\vert\, \left\langle v,w\right\rangle_{\R^{n+2}}=0\right\},
\end{align*}
where $ \left\langle\cdot,\cdot\right\rangle_{\R^{n+2}}$ denotes the standard metric of $\R^{n+2}$.
\end{definition}

Clearly, the Stiefel manifold can be identified with
\begin{align*}
\left\{z\in \Sph^{2n+3}\vert\, \left\langle z,z\right\rangle_{\C^{n+2}}=0\right\}\subset\Sph^{2n+3}
\end{align*}
by the map $(v,w)\mapsto \tfrac{1}{\sqrt{2}} (v+iw)\in \mathbb{S}^{2n+3}$.
Consequently, $\pi(St_2(\R^{n+2}))=Q_n$ and it follows easily that the projection $\pi:\,St_2(\R^{n+2})\rightarrow Q_n\subset\CP^{n+1}$ is a Riemannian submersion.

\smallskip

In order to associate a Lagrangian submanifold to an isoparametric hypersurface $(F_t,\nu_t)$, we first lift the embeddings $F_t$ to the Stiefel manifold $St_2(\mathbb{R}^{n+2})$ and then concatenates this map with the 
projection onto the Grassmannian $\mbox{Gr}_2^+(\R^{n+2})$.
For $t\neq\theta_j$ we define the map $\hat{F}_{t}$ by
\begin{align*}
\hat{F}_{t}:\,M\rightarrow St_2(\mathbb{R}^{n+2}),\hspace{0.5cm}  p\mapsto\hat{F}_{t}(p):=\tfrac{1}{\sqrt{2}}(F_{t}(p)+i{\nu_{t}}_{\vert_p}).
\end{align*}
Furthermore, we introduce the map $\frak{L}$ by
$$\frak{L}:=\pi\circ\hat{F}_{t}:M\rightarrow Q_n,\hspace{0.5cm} p\mapsto \lbrack\hat{F}_0(p)\rbrack.$$ 

\begin{figure}[h]
\begin{center}
\begin{tikzpicture}
\draw[->,thick](0:0)--(0:3cm);
\node at (0:-0.5cm) {$M^n$};
\node at (1.5cm,0.3cm) {$\hat{F_t}$};
\node at (0:4.1cm) {$St_2(\R^{n+2})$};
\draw[->,thick](4.1cm,-0.3cm)--(4.1cm,-1.75cm);
\node at (4.4cm,-1.05cm) {$\pi$};
\draw[->,thick](-0.3cm,-0.3cm)--(3.8cm,-1.9cm);
\node at (5cm,-2.0cm) {$Q^n\subset\CP^{n+1}$};
\node at (1.8cm,-1.4cm) {$\frak{L}$};
\end{tikzpicture}
\label{fig2}
\caption{The Lagrangian immersion $\frak{L}$}
\end{center}
\end{figure}

Note that $\frak{L}(p)$ is by construction the oriented $2$-plane in $\R^{n+2}$ which is spanned by ${F}_{t}(p)$ and $\nu_{t\lvert p}.$
Since we have $\hat{F_t}=e^{-it}\hat{F}_0$, the immersion $\frak{L}$ does not depend on the parameter $t$.
Thus we obtain the following lemma.

\begin{lemma}
Let a family of isoparametric hypersurfaces in a sphere be given. For each isoparametric hypersurface in this family we obtain the same
immersion $\frak{L}$.
\end{lemma}

As already mentioned above, the next result was first proved by Palmer \cite{palm1, palm2}.
For convenience of the reader we reprove this result.

\begin{proposition}[\cite{palm1, palm2}]
\label{palmer}
The map $\frak{L}:=\pi\circ\hat{F}_{t}:M\rightarrow Q_n$, $p\mapsto \lbrack\hat{F}_0(p)\rbrack$, is a Lagrangian immersion
with normal vector field $N_Z=\pm i\,d\hat{F}_0Z$, where $Z\in TM.$
\end{proposition}
\begin{proof}
Using the identity $d\nu_{t}=-dF_{t}A_{t}$, we obtain
\begin{align*}
d\hat{F}_{t}=\tfrac{1}{\sqrt{2}}\,dF_{t}(\eins-iA_{t})
\end{align*}
for all $t\in\R$. Thus we get
\begin{align*}
\hat{F_t}^{*}\omega(X,Y)&=g_Q(J d\hat{F_t}X,d\hat{F_t}Y)=\mbox{Re}{\left\langle id\hat{F}_{0}X,d\hat{F}_{0}Y\right\rangle}_h\\&=\tfrac{1}{2}
\big(g_0(A_{0}X,Y)+g_0(X,-A_{0}Y)\big)=0,	
\end{align*}
which proves our claim.
\auxqed
\end{proof}

\begin{lemma}
\label{lhz}
$\hat{F}_{t}(M)$ is horizontal with respect to the projection $\pi:St_2(\R^{n+2})\rightarrow Q_n.$
\end{lemma}
\begin{proof}
For any $p\in M$
\begin{align*}
\mbox{Im}(dF_{0\vert p})\perp\Span\left\{F_{0}(p),\nu_{0\vert p}\right\}
\end{align*}
in $\R^{n+1}$. Thus identities (\ref{abf}) and (\ref{abn}) imply
 \begin{align*}
\mbox{Im}(d\hat{F}_{t\vert p})\perp \C\hat{F}_{t}(p)\oplus\C\overline{\hat{F}_{t}(p)}
\end{align*}
in $\C^{n+1}$. The claim now follows from (\ref{hz}).
\end{proof}

\begin{remark}
The above construction was used by H. Ma and Y. Ohnita in \cite{ohnita} to classify compact homogeneous Lagrangian submanifolds in complex hyperquadrics.
\end{remark}

\subsection{A set of invariants of the Lagrangian submanifold}
In this subsection we introduce a set of invariants for the Lagrangian submanifold of the complex quadric.
Furthermore, for each of these invariants, we establish some basic properties.

\subsubsection{The metric $\hat{g}$}
We endow the Lagrangian submanifold with the natural metric, i.e., the Riemannian metric $\hat{g}$ on $M$ induced from $\left\langle \cdot,\cdot\right\rangle_h$ by $\hat{F}_{t}$.
This metric is given by
\begin{align*}
\hat{g}=\mbox{Re} (\hat{F}_{t}^{*}\left\langle \cdot,\cdot\right\rangle_h).
\end{align*}
Lemma\,\ref{lhz} asserts that the Riemannian metric $\hat{g}$ is induced from $g_Q$ by $\frak{L}$.
In particular, the metric $\hat{g}$ is independent of $t.$

\smallskip

In the next two lemmas we relate $\hat{g}$ and the associated Levi-Civita connection to $g_t$ and $\nabla^t$, respectively.

\begin{lemma}
\label{independent}
For each $p\in M$ and all $t\in\R$ we have $$\hat{g}(X,Y)={g_{t}}(X,\tfrac{1}{2}(\eins+{A_{t}}^2)Y)=
{g_{0}}(X,\tfrac{1}{2}(\eins+{A_{0}}^2)Y)\,\,\mbox{for all}\,\,X,Y\in\Gamma(TM).$$ In particular $\hat{g}$ is independent of $t$.
\end{lemma}
\begin{proof}
Since $\hat{F}_{t}=e^{-it}\hat{F}_{0}$ we have
\begin{align*}
d\hat{{F}}_{t}=e^{-it}d\hat{F}_{0}.
\end{align*}
By the definition of $\hat{g}$ and the preceding identity this gives
\begin{align*}
\hat{g}(X,Y)=\mbox{Re}\left\langle d\hat{F}_{t}X,d\hat{F}_{t}Y\right\rangle_h=\mbox{Re}\left\langle d\hat{F}_{0}X,d\hat{F}_{0}Y\right\rangle_h,
\end{align*}
for all $X,Y\in\Gamma(TM)$, and therefore
\begin{align*}
\hat{g}(X,Y)=g_{t}(X,\tfrac{1}{2}(\eins+A_{t}^2)Y)=g_{0}(X,\tfrac{1}{2}(\eins+A_{0}^2)Y).
\end{align*}
\auxqed
\end{proof}

\begin{remark}
The induced metric $\hat{g}$ is the arithmetic mean of some $g_{t}$:\\ 
let $\phi\in\left(0,\pi/2g\right)$ be given and define the arc $\xi_k$ for $k\in\mathbb{N}$ by
$\xi_k=\phi+k\pi/2g$.
Then $\hat{g}=\frac{1}{2}(g_{\phi}+g_{\phi+\pi/2})$ and $\hat{g}=\frac{1}{2g}\sum_{k=0}^{2g-1}{g_{\xi_k}}$.
\end{remark}

We denote the Levi-Civita connection associated to $\hat{g}$ by $\nabla$. The next lemma relates $\nabla$ and $\nabla^t$.

\begin{lemma}
\label{nik2}
The connections $\nabla$ and $\nabla^{t}$ are related by
\begin{align*}
\nabla_XY=\nabla_X^{t}Y+A_{t}(\eins+A_{t}^2)^{-1}(\nabla_X^{t}A_{t})Y.
\end{align*}
\end{lemma}
\begin{proof}
The Koszul formula yields
\begin{align*}
2\hat{g}(\nabla_XY,Z)=(\nabla^t_X\hat{g})(Y,Z)+(\nabla^t_Y\hat{g})(X,Z)-(\nabla^t_Z\hat{g})(X,Y)+2\hat{g}(\nabla^t_XY,Z).
\end{align*}
By Lemma\,\ref{independent} we get 
$$(\nabla^t_{X_1}\hat{g})(X_2,X_3)=\tfrac{1}{2}g_t(\big((\nabla^t_{X_1}A_t)A_t+A_t(\nabla^t_{X_1}A_t)\big)X_2,X_3).$$
Consequently, using $(\nabla^t_{X_1}A_t)X_2=(\nabla^t_{X_2}A_t)X_1$, we have
\begin{align*}
(\nabla^t_X\hat{g})(Y,Z)+(\nabla^t_Y\hat{g})(X,Z)-(\nabla^t_Z\hat{g})(X,Y)=g_t(A_t(\nabla^t_XA_t)Y,Z).
\end{align*}
Hence we obtain
\begin{align*}
\hat{g}(\nabla_XY-\nabla^t_XY,Z)=\tfrac{1}{2}g_t(A_t(\nabla^t_XA_t)Y,Z)=\hat{g}(A_t(\eins+A_{t}^2)^{-1}(\nabla_X^{t}A_{t})Y,Z),
\end{align*}
where the last equality follows from Lemma\,\ref{independent}. Since this identity holds for arbitrary $Z\in\Gamma(TM)$, the claim is established.
\end{proof}

\subsubsection{The tensor $\alpha$}
Throughout this subsection we assume $t\in\left[0,2\pi\right]\cap\mathcal{R}$, see the subsection \lq Global structure\rq\, of Section\,1 for the definition of the regular set $\mathcal{R}$.
The symmetric tensor $\alpha^t$, which is given by the formula
\begin{align*}
\alpha^t(X,Y,Z)=g_t((\nabla_X^tA_t)Y,Z),
\end{align*}
is one of the fundamental invariants used in the previous classification approaches though
usually encoded in some much less invariant Maurer-Cartan forms $\Lambda_t$.
The really interesting fact is that $\alpha$ coincides - up to a factor of two - with the 
second fundamental form of the Lagrangian submanifold, which gives a new, 
geometric interpretation for this important tensor.
When formulating identities in terms of $\alpha^t$ in Section \ref{sec3}, further advantages of working with $\alpha^t$ instead of $\Lambda_t$ will become obvious.
In this subsection we establish some basic properties of $\alpha^t$.

\begin{definition}
\label{defa}
For $t\in\left[0,2\pi\right]\cap\mathcal{R}$
define $\alpha^t:\Gamma(TM)^3\rightarrow C^{\infty}(M,\mathbb{R})$ by
\begin{align*}
\alpha^t(X,Y,Z)=g_{t}((\nabla^t_{X}A_t)Y,Z),
\end{align*} 
for $X,Y,Z\in\Gamma(TM)$.
\end{definition}

\begin{lemma}
\label{van}
The map $\alpha^t:\Gamma(TM)^3\rightarrow C^{\infty}(M,\mathbb{R})$ is symmetric and trilinear.
Furthermore, for each $j\in\left\{1,...,g\right\}$ the restriction
$\alpha^t:D_j\times D_j\times TM \rightarrow \mathbb{R}$ is zero. In particular the map $\alpha^t$ is trace free.
\end{lemma}
\begin{proof}
The tensor $\alpha^t$ is obviously trilinear. Since $M$ is a hypersurface in a constant curvature space,
the Codazzi equation states that
\begin{align*}
(\nabla^t_{X} A_t)Y=(\nabla^t_{Y} A_t)X.
\end{align*}
Hence $\alpha^t$ is symmetric as well.\\
Next we prove that $\alpha^t$ vanishes when we choose two of its entries to be in the same distribution.
Since $\alpha^t$ is symmetric we can assume without loss of generality that  $Y,Z\in D_j$ for a $j\in\left\{1,...,g\right\}$ and $X\in\Gamma(TM)$. Thus we get
\begin{align*}
\alpha^t(X,Y,Z)=g_{t}((\nabla^t_{X}A_t)Y,Z)=g_{t}(\nabla^t_{X}Y,(\lambda^t_j-A_t)Z)=0,
\end{align*}
which establishes the claim.
\auxqed
\end{proof}
 
Next we endow $\alpha^{t}$ with a geometric meaning by proving that $\alpha^{t}$ is - up to a constant factor - given by
the second fundamental form of $\frak{L}(M)\subset Q^n$. In particular, $\alpha^t$ is independent of $t\in\mathcal{R}$. In other words, 
we associate to each isoparametric hypersurface in a family of isoparametric hypersurface to same tensor.

\smallskip

We denote by $\hat{A}:\nu M\rightarrow \mbox{End}(\Gamma(TM))$ the shape operator the submanifold $(M,\hat{g})$ of $(Q^n,g_Q)$.
Furthermore, let  $\hat{\alpha}:\Gamma(TM)^3\rightarrow C^{\infty}(M,\mathbb{R})$ be the second fundamental form of the Lagrangian submanifold $\frak{L}(M)\subset Q^n$, i.e.,
\begin{align*}\hat{\alpha}(X,Y,Z)=\hat{g}({\hat{A}}_{N_{X}}Y,Z)
\end{align*} for $X,Y,Z\in\Gamma(TM)$.
Recall, that $N_X$ denotes the normal vector field introduced in Proposition\,\ref{palmer}. 
The next theorem establishes Theorem\,A of the introduction.

\begin{theorem}
\label{alpha}
For any $t\in \mathcal{R}$, the maps $\hat{\alpha}:\Gamma(TM)^3\rightarrow C^{\infty}(M,\mathbb{R})$ and $\alpha^t:\Gamma(TM)^3\rightarrow C^{\infty}(M,\mathbb{R})$ are related by
\begin{align*}
2\hat{\alpha}=\alpha^t.
\end{align*}
In particular the map $\alpha^t$ is independent of $t\in\mathcal{R}.$
\end{theorem}
\begin{proof}
Throughout the proof fix $X,Y,Z\in\Gamma(TM)$.
Furthermore, we use the convention $N_Z=-i\,d\hat{F_0}Z$.
By definition of $\hat{A}$ and skew symmetry of $J$ we get
\begin{align*}
\hat{g}(\hat{A}_{N_Z}X,Y)=\mbox{Re}\langle \nabla^{Q^n}_XN_Z,d\hat{F}_tY\rangle_h&=-\mbox{Re}\langle J\nabla^{Q^n}_X(d\hat{F}_tZ),d\hat{F}_tY\rangle_h\\&=\mbox{Re}\langle \nabla^{Q^n}_X(d\hat{F}_tZ),J d\hat{F}_tY\rangle_h.
\end{align*}
Moreover, we have
\begin{align}
\label{sit}
\mbox{Re}\langle \nabla^{Q^n}_Xd\hat{F}_tZ,J d\hat{F}_tY\rangle_h=
\mbox{Re}\langle \nabla^{St}_Xd\hat{F}_tZ,J d\hat{F}_tY\rangle_h=
\mbox{Re}\langle d_Xd\hat{F}_tZ,J d\hat{F}_tY\rangle_h,
\end{align}
 where $\nabla^{St}$ and $d$ denote the Levi-Civita connection of the Stiefel manifold and Euclidean space, respectively.
Indeed, the first equality holds since the image of $\hat{F}_{t}$ is horizontal with respect to the projection $\pi:St_2(\R^{n+2})\rightarrow Q_n$, see Lemma\,\ref{lhz}, and $\pi$ is a Riemannian submersion.
The second equality simply follows since $St$ is contained in the Euclidean space.

\smallskip

Plugging $d\hat{F}_t=\frac{1}{\sqrt{2}}dF_t(\eins-iA_t)$ into (\ref{sit}), we obtain
\begin{align*}
\mbox{Re}\langle \nabla^{Q^n}_Xd\hat{F}_tZ,J d\hat{F}_tY\rangle_h=\tfrac{1}{2}\langle d_X(dF_tZ),dF_tA_tY\rangle_{\R^{n+2}}-\tfrac{1}{2}\langle d_X(dF_tA_tZ),dF_tY\rangle_{\R^{n+2}}.
\end{align*}
Since the Weingarten equation is given by
\begin{align*}
d_X(dF_tZ)=dF_t\nabla^t_XZ+\langle A_tX,Z\rangle\nu,
\end{align*}
we get
\begin{align*}
\mbox{Re}\langle \nabla^{Q^n}_X(d\hat{F}_tZ),N_Y\rangle_h&=-\mbox{Re}\langle \nabla^{Q^n}_Xd\hat{F}_tZ,J d\hat{F}_tY\rangle_h\\
&=-\tfrac{1}{2}g_t(\nabla^t_XZ,A_tY)+\tfrac{1}{2}g_t(\nabla^t_X(A_tZ),Y)\\&=\tfrac{1}{2}\,g_t((\nabla^t_XA_t)Y,Z)=\tfrac{1}{2}\alpha^t(X,Y,Z).
\end{align*}
Using the identity
\begin{align*}
\mbox{Re}\langle d_Xd\hat{F}_tZ,J d\hat{F}_tY\rangle_h=\mbox{Re}\langle d_Xd\hat{F}_0Z,J d\hat{F}_0Y\rangle_h,
\end{align*}
an analogous calculation yields 
\begin{align*}
\mbox{Re}\langle d_Xd\hat{F}_0Z,J d\hat{F}_0Y\rangle_h=\tfrac{1}{2}\alpha^0(X,Y,Z).
\end{align*}
Thus we in particular get
$\alpha^0(X,Y,Z)=\alpha^t(X,Y,Z)$,
which proves the claim. 
\end{proof}

Since $\alpha^t$ is independent of $t\in\mathcal{R}$, we will denote this tensor henceforth simply by $\alpha$.

\subsection{The invariant $B\otimes B^{-1}$}
\label{spro}
In this section we assign to each isoparametric hypersurface $(M,g_t)$ an operator $B_t$, and show that $B_t\otimes B_t^{-1}$ is independent of $t$.

\begin{definition}
Let $t\in\left[0,2\pi\right]\cap\mathcal{R}$. Define $B_{t}:\Gamma(TM)\rightarrow \Gamma(TM)\otimes\C$ via
\begin{align*}
B_{t}=(A_{t}+i\eins)\,(A_{t}-i\eins)^{-1}.
\end{align*}
\end{definition}

In what follows we denote by $\hat{g}$ also the complex bilinear extension of $\hat{g}$.
By the very definition of $B_t$ we get the following lemma.

\begin{lemma}
$\hat{g}(B_tX,Y)=-(\hat{F}_t^{*}q)(X,Y).$
\end{lemma}

In other words $B_t$ encodes the metric $\hat{F}_t^{*}q$ and thus arises as a natural invariant of the Lagrangian submanifold $\frak{L}(M)\subset Q^n$.

\begin{lemma}
\label{bthe}
The operators $B_{t}$ are trace free and satisfy the identities
\begin{align*}
B_{t}^g=-e^{-2git}\eins,\hspace{0.5cm} B_{t}^{-1}=-e^{2git}B_t^{g-1},\hspace{0.5cm}\overline{B_{t}}=B_{t}^{-1},\hspace{0.5cm}B_{t+\phi}=e^{-2i\phi}B_{t}\hspace{0.2cm}\forall\,\phi\in \R.
\end{align*}
\end{lemma}
\begin{proof}
Every $X\in D_j$ is an eigenvector of $B_{t}$ with 
eigenvalue $\mu_{j}^{t}\in\C$ given by $\mu_{j}^{t}=e^{2i(\theta_{j}-t)}$.
Using the special form of $\theta_j$, we obtain $B_t^g=-e^{-2git}\eins$.
The second identity is an immediate consequence of the first identity.
Moreover, the third equation follows from the definition of $B_t$.
Finally, an argument analogous to the proof of Lemma\,\ref{shapeoperator} gives the identity
\begin{align*}
{A_{t+\phi}}=(\eins+\cot(\phi){A_{t}})(\cot(\phi)\eins-{A_{t}})^{-1}.
\end{align*}
and hence the fourth identity follows from the very definition of $B_t$.
\end{proof}

At first glance it might appear wrong to work with the operator $B_t$
since it depends on the parameter $t$. As it turns out, however, all relevant
identities factor through the operator $B_t\otimes B_t^{-1}$, which is
independent of $t$.

\begin{corollary}
The expression $B_t\otimes B_t^{-1}$ is independent of $t\in\R$.
\end{corollary}
\begin{proof}
By the last identity of Lemma\,\ref{bthe} we get $B_t=e^{-2it}B_0$, which implies $B_t^{-1}=e^{2it}B_0^{-1}$.  Hence, $B_t\otimes B_t^{-1}=B_0\otimes B_0^{-1}$.
\end{proof}

The preceding corollary allows us to introduce the tensor $B\otimes B^{-1}:=B_t\otimes B_t^{-1}$ for some $t\in\R$.

\smallskip

In the next lemma we express the projections on the distributions in terms of $B_{t}$. 
\begin{lemma}
\label{pro}
The projector $\pi_{j}:M\rightarrow D_j\subset M$ is given by
\begin{align*}
\pi_{j}=\tfrac{1}{g}\sum_{k=0}^{g-1}{B_{\theta_j}^k}.
\end{align*}
\end{lemma}
\begin{proof}
Let $X\in D_m$ be given. Using $\theta_l\,\equiv(2(\,l-1)+1)\,\frac{\pi}{2g}\,\,\mbox{mod}\, \pi$ we get
\begin{align*}
\pi_{j}X =\tfrac{1}{g}\sum_{k=0}^{g-1}{B_{\theta_j}^k}X=\tfrac{1}{g}\sum_{k=0}^{g-1}{e^{2i(\theta_m-\theta_j)k}X}=\delta_{m,j}X.
\end{align*}
\auxqed
\end{proof}

From Lemma\,\ref{van} we have that $\alpha(X,Y,Z)$ vanishes if two of its entries lie in the same distribution, i.e. we have
$\alpha(\pi_kX,\pi_kY,Z)=0$ for each 
$k\in\{1,\ldots,g\}$.
In the next corollary we express the latter condition in terms of $B\otimes B^{-1}$.

\begin{corollary}
The condition $\alpha(\pi_kX,\pi_kY,Z)=0$,  where $k\in\{1,\ldots,g\}$, is equivalent to the identity
\begin{align*}
\sum_{j=0}^{g-1}\alpha(B^jX,B^{-j}Y,Z)=0.
\end{align*}
\end{corollary}
\begin{proof}
Apply Lemma\,\ref{pro} to the identity  $\alpha(\pi_kX,\pi_kY,Z)=0$ and sum over $k$.
\end{proof}

Summarizing the results of the present section,
 we assign to each isoparametric hypersurface in a given family of isoparametric hypersurfaces a set of invariants
$$(\hat{g},\alpha,B\otimes B^{-1}),$$ 
which depends only on the isoparametric family it is contained in.

\section{Weyl and Symmetry identities}
\label{sec3}
In the present section we formulate all relevant identities in terms of the invariants
$\hat{g}, \alpha$ and $B\otimes B^{-1}$. We in particular, reveal the importance of the Weyl Identities and explain
how they enter the existing classification approaches.

\smallskip

In the first subsection we establish the fundamental submanifold equations for the Lagrangian submanifold of the complex quadric, i.e. the Codazzi, Gauss and Ricci equations.
In the second and third subsection we deduce the Weyl Identities and the Symmetry Identities, respectively.

\subsection{The Codazzi, Gauss and Ricci equations}
In this subsection we provide the Gauss equation, the Codazzi equation and the Ricci equation for the
submanifold $(M,\hat{g})\subset(Q^n,g_Q)$. Let $X,Y,Z,W\in \Gamma(TM)$ throughout.

\smallskip

For ease of notation, we introduce $T:\Gamma(TM)^4\rightarrow C^{\infty}(M,\mathbb{R})$
 by
\begin{align*}
T(X,Y,Z,W)=\Im(\hat{g}(B_{0}X,Y)\hat{g}(B_{0}^{-1}Z,W)),
\end{align*}
and the $(2,0)$-tensors $b$ and $ \overline{b}$ by
\begin{align*}
b(X,Y):=\hat{g}(B_0X,Y)\quad\mbox{and}\quad \overline{b}(X,Y):=\hat{g}(B_0^{-1}X,Y). 
\end{align*}
Furthermore, recall
\begin{align*}
(\alpha\owedge_{\hat{g}}\alpha)(X,Y,Z,W)=\tr_{\hat{g}}\big(\alpha(X,W,\,\cdot\,)\alpha(Y,Z,\,\cdot\,)-\alpha(X,Z,\,\cdot\,)\alpha(Y,W,\,\cdot\,)\big).
\end{align*}

In terms of this notation, the Codazzi and the Gauss equation take an easy form.

\begin{proposition}
\label{cod}
The Codazzi and the Gauss equation of the submanifold $(M,\hat{g})\subset(Q^n,g_Q)$ are given by
\begin{enumerate}
\item $(\nabla_{X}\alpha)(Y,Z,W)-(\nabla_{Y}\alpha)(X,Z,W)=2\left[T(X,Z,Y,W)+T(X,W,Y,Z)\right]$, \mbox{and}
\item $R(X,Y,Z,W)=(\hat{g}\owedge \hat{g}+b\owedge\overline{b}+\frac{1}{4}\alpha\owedge_{\hat{g}}\alpha)(X,Y,Z,W)$,
\end{enumerate}
respectively.
\end{proposition}

\begin{proof}
We start by proving that the Codazzi equation for $(M,\hat{g})\subset(Q^n,g_Q)$ is given by the first identity.
Recall the Codazzi equation
\begin{align*}
\hat{g}((\nabla_{X}\hat{A})_{N_{Z}}Y,W)-\hat{g}((\nabla_{Y}\hat{A})_{N_{Z}}X,W)=&R^Q(d\hat{F_t}X,d\hat{F_t}Y,d\hat{F_t}W,N_Z)\\=&-R^Q(d\hat{F_t}X,d\hat{F_t}Y,d\hat{F_t}W,J d\hat{F_t}Z).
\end{align*}
Using Theorem\,\ref{alpha}, the left and side simplifies to 
$$\frac{1}{2}\big((\nabla_{X}\alpha)(Y,Z,W)-(\nabla_{Y}\alpha)(X,Z,W)\big).$$ 
Thus it remains to prove that he right hand side is given by
$$2R^Q(d\hat{F_t}X,d\hat{F_t}Y,d\hat{F_t}Z,J d\hat{F_t}W)=-\big(T(X,Z,Y,W)+T(X,W,Y,Z)\big).$$
By a straightforward calculation we get
\begin{align*}
(g_Q\owedge g_Q+\omega\owedge\omega)(d\hat{F_t}X,d\hat{F_t}Y,d\hat{F_t}Z,Jd\hat{F_t}W)=0.
\end{align*}
Thus by Lemma\,\ref{hrq} we have
\begin{align*}
R^Q(d\hat{F_t}X,d\hat{F_t}Y,d\hat{F_t}Z,J d\hat{F_t}W)=q\owedge\overline{q}(d\hat{F_t}X,d\hat{F_t}Y,d\hat{F_t}Z,J d\hat{F_t}W).
\end{align*}
Since
\begin{align*}
&q(d\hat{F_t}X_1,d\hat{F_t}X_2)=-\hat{g}(X_1,B_tX_2), &q(d\hat{F_t}X_1,J d\hat{F_t}X_2)=-i\hat{g}(X_1,B_tX_2),\\& \overline{q}(d\hat{F_t}X_1,d\hat{F_t}X_2)=-\hat{g}(X_1,B_t^{-1}X_2), &\overline{q}(d\hat{F_t}X_1,Jd\hat{F_t}X_2)=i\hat{g}(X_1,B_t^{-1}X_2),
\end{align*}
for all $X_1,X_2\in\Gamma(TM)$, an easy calculation yields the result.

\smallskip

In order to prove the second identity, recall that
the Gauss equation for $(M,\hat{g})\subset (Q_n,g_Q)$ is given by
\begin{align*}
R(X,Y,Z,W)&=R^{Q}(d\hat{F_t}X,d\hat{F_t}Y,d\hat{F_t}Z,d\hat{F_t}W)\\&\,\,\,\,\,\,+g_Q(\Pi(X,W),\Pi(Y,Z))-g_Q(\Pi(X,Z),\Pi(Y,W)),
\end{align*}
where $\Pi$ denotes the second fundamental form of $(M,\hat{g})\subset (Q_n,g_Q)$.
Furthermore, recall that we have $R^{Q}=g_Q\owedge g_Q+\omega\owedge\omega+q\owedge\overline{q}$ by Lemma \ref{hrq}.
A straightforward calculation yields 
\begin{align*}
&q\owedge\overline{q}(d\hat{F_t}X,d\hat{F_t}Y,d\hat{F_t}Z,d\hat{F_t}W)=b\owedge\overline{b}(X,Y,Z,W),\\
&g_Q\owedge g_Q(d\hat{F_t}X,d\hat{F_t}Y,d\hat{F_t}Z,d\hat{F_t}W)=\hat{g}\owedge \hat{g}(X,Y,Z,W).
\end{align*}
Furthermore, since $(M,\hat{g})$ is a Lagrangian submanifold of $(Q_n,g_Q)$ we get
\begin{align*}
\omega\owedge\omega(d\hat{F_t}X,d\hat{F_t}Y,d\hat{F_t}Z,d\hat{F_t}W)=0.
\end{align*}
Combining these equalities we obtain
\begin{align*}
R(X,Y,Z,W)&=(\hat{g}\owedge \hat{g}+b\owedge\overline{b})(X,Y,Z,W)\\&\,\,\,\,\,\,+g_Q(\Pi(X,W),\Pi(Y,Z))-g_Q(\Pi(X,Z),\Pi(Y,W)).
\end{align*}
One can naturally assign to every $\hat{g}$-orthonormal basis $(e_i)_{i=1}^n$ of $TM$ a $g_Q$-orthonormal basis of $\nu(TM)$, namely $(J d\hat{F}_0e_i)_{i=1}^n.$
Hence we arrive at the identity 
\begin{align*}
R(X,Y,Z,W)&=(\hat{g}\owedge \hat{g}+b\owedge\overline{b})(X,Y,Z,W)\\&\,\,\,\,\,\,+\sum_{i=1}^n{g_Q(\Pi(X,W),J d\hat{F}_0e_i)g_Q(J d\hat{F}_0e_i,\Pi(Y,Z))}\\&\,\,\,\,\,\,-\sum_{i=1}^n{g_Q(\Pi(X,Z),J d\hat{F}_0e_i)g_Q(J d\hat{F}_0e_i,\Pi(Y,W))}.
\end{align*}
Using $g_Q(\Pi(X,Y),\xi)=\hat{g}(\hat{A}_{\xi}X,Y)$ for $\xi\in\nu(TM)$ and
Theorem\,\ref{alpha} we obtain the desired result.
\end{proof}

\begin{remark}
\label{ricci}
The Ricci equation of the Lagrangian submanifold $(M,g)\subset (Q_n,g_Q)$ is equivalent to the Gauss equation of $(M,g)\subset (Q_n,g_Q).$ 
\end{remark}

\subsection{The Weyl Identity}
In this subsection we first recall the classical Weyl Identities.
Afterwards we provide the invariant Weyl Identity.
Finally, we will explain the importance of the Weyl Identities.

\subsubsection{The classical Weyl Identities}
In \cite{karcher} Karcher deduced the so-called Weyl Identities, which he descirbes as \lq relations between the principal curvatures and the covariant derivatives of the shape operator derived by differentiating the Codazzi equations and combining with the Gauss equations.\rq\ 
These identities, which are henceforth referred to as the classical Weyl Identities, are stated in the following theorem.

\begin{theorem}[\cite{karcher}]
\label{clwe}
For all $i,j\in\left\{1,...,g\right\}$ with $i\neq j$ we have
\begin{align*}
(1+\lambda_i\lambda_j)\,g_0(v_i,v_i)\,g_0(v_j,v_j)=2g_0((\nabla^0_{v_i}A_0)v_j,(\lambda_i-A_0)^{-1}(\lambda_j-A_0)^{-1}(\nabla^0_{v_i}A_0)v_j),
\end{align*}
where $v_i\in D_i$, $v_j\in D_j$ and $\lambda_m=\lambda_m(0)$.
\end{theorem}

By polarizing the preceding identity twice and expressing the resulting equation in terms of $\alpha$
 we obtain the following corollary.

\begin{corollary}
\label{wij2}
For all $i,j\in\left\{1,...,g\right\}$ with $i\neq j$, the identity
\begin{align*}
(1+\lambda_i\lambda_j)\,g_0(v_i,\widetilde{v_i})\,g_0(v_j,\widetilde{v_j})=\tr_{g_0}\big(&\alpha(v_i,v_j,\cdot)\,\alpha(\widetilde{v_i},\widetilde{v_j},(\lambda_i-A_0)^{-1}(\lambda_j-A_0)^{-1}\,\cdot\,)\\+&
\alpha(\widetilde{v_i},v_j,\,\cdot\,)\,\alpha(v_i,\widetilde{v_j},(\lambda_i-A_0)^{-1}(\lambda_j-A_0)^{-1}\,\cdot\,)
\big)
\end{align*}
is equivalent to the Weyl Identity, where $v_i,\widetilde{v_i}\in D_i,$ $v_j,\widetilde{v_j}\in D_j$, $\lambda_m=\lambda_m(0)$ and $\tr_{g_0}$ denotes the sum over a $g_0$-orthonormal basis of
the orthogonal complement in $TM$ to $D_j\oplus D_i$.
\end{corollary}

\begin{remark}
The classical Weyl Identities depend on several indices.
Taking higher covariant derivatives of these identities would consequently lead to a plethora of different cases.
The importance of the higher covariant derivatives of the Weyl Identities is explained in Subsection \ref{weyl}.
\end{remark}

\subsubsection{Invariant Weyl Identity}
The classical Weyl Identities depend on several indices. 
In terms of the invariants introduced in Section \ref{sec2}, these multiple identities can be expressed as a single tensor identity, which we shall call the \emph{invariant Weyl Identity}. 

\smallskip

In this subsection we provide the invariant Weyl Identity.
As preparation we establish the following two lemmas.

\begin{lemma}
\label{ableitungb}
$\hat{g}((\nabla_{X}B_t)Y,Z)=-\frac{i}{2}\big(\alpha (X,B_tY,Z)+\alpha(X,Y,B_tZ)\big).$
\end{lemma}
\begin{proof}
Due to the last identity of Lemma\,\ref{bthe} it is sufficient to prove the claim for $t=0$.
By definition of $B_0$ we obtain
\begin{align*}
(\nabla_{X}B_0)Y=-2i(A_0-i\eins)^{-1}(\nabla_{X}A_0)(A_0-i\eins)^{-1}Y.
\end{align*}
Furthermore, using Lemma\,\ref{nik2} we get
\begin{align*}
(\nabla_{X}A_0)Y=(\nabla^0_{X}A_0)Y+A_0(\eins+A_0^2)^{-1}((\nabla^0_{X}A_0)A_0Y-A_0(\nabla^0_{X}A_0)Y).
\end{align*}
Consequently, we obtain
\begin{align*}
\hat{g}((\nabla_{X}B_0)Y,Z)=&-ig_0((A_0+i\eins)(\nabla_{X}A_0)(A_0-i\eins)^{-1}Y,Z)\\
=&-ig_0((\nabla^0_{X}A_0)(A_0-i\eins)^{-1}Y,(A_0+i\eins)Z)\\
&-ig_0(\big(A_0(\eins+A_0^2)^{-1}(\nabla^0_{X}A_0)(A_0(A_0-i\eins)^{-1}Y)\big),(A_0+i\eins)Z)\\
&+ig_0(\big(A_0(\eins+A_0^2)^{-1}A_0(\nabla^0_{X}A_0)(A_0-i\eins)^{-1}Y\big),(A_0+i\eins)Z).
\end{align*}
Expressed in terms of $\alpha$, this equation reads
\begin{align*}
\hat{g}((\nabla_{X}B_0)Y,Z)=&-i\alpha(X,(A_0-i\eins)^{-1}Y,(A_0+i\eins)Z)\\
&-i\alpha(X,A_0(A_0-i\eins)^{-1}Y,A_0(A_0-i\eins)^{-1}Z),
\end{align*}
where we make use of
\begin{align*}
\eins-A_0^2(\eins+A_0^2)^{-1}=(\eins+A_0^2)^{-1}\,\,\,\mbox{and}\,\,\,(\eins+A_0^2)^{-1}(A_0+i\eins)=(A_0-i\eins)^{-1}.
\end{align*}
By definition of $B_0$ we get
\begin{align*}
\tfrac{1}{2}(B_0-\eins)=i(A_0-i\eins)^{-1}\,\,\,\mbox{and}\,\,\,\tfrac{1}{2}(B_0+\eins)=A_0(A_0-i\eins)^{-1}.
\end{align*}
Hence we find 
\begin{align*}
\hat{g}((\nabla_{X}B_0)Y,Z)=-\tfrac{i}{2}\big(\alpha(X,B_0Y,Z)+\alpha(X,Y,B_0Z)\big).
\end{align*}
Using $B_t=e^{-2it}B_0$ the claim is thus established.
\end{proof}

\begin{lemma}
\label{pwe}
We have the identity
\begin{align*}
2T(\pi_kX,\pi_kY,\pi_jZ,\pi_jW)=&-\alpha(\pi_kY,(\nabla_{\lower 2pt\hbox{$\scriptstyle\pi_k X$}}\pi_j)Z,\pi_jW)-\alpha(\pi_kY,\pi_jZ,(\nabla_{\lower 2pt\hbox{$\scriptstyle\pi_k X$}}\pi_j)W)\\
&+\alpha(\pi_jW,(\nabla_{\lower 2pt\hbox{$\scriptstyle\pi_j Z$}}\pi_k)X,\pi_kY)+\alpha(\pi_jW,\pi_kX,(\nabla_{\lower 2pt\hbox{$\scriptstyle\pi_j Z$}}\pi_k)Y).
\end{align*}
\end{lemma}
\begin{proof}
 Differentiating the equation $\alpha(Y,\pi_jZ,\pi_jW)=0$ we get
\begin{align*}
(\nabla_{\lower 2pt\hbox{$\scriptstyle X$}}\alpha)(Y,\pi_jZ,\pi_jW)+\alpha(Y,(\nabla_{\lower 2pt\hbox{$\scriptstyle X$}}\pi_j)Z,\pi_jW)+\alpha(Y,\pi_jZ,(\nabla_{\lower 2pt\hbox{$\scriptstyle X$}}\pi_j)W)=0.
\end{align*}
Consequently we obtain
\begin{multline*}
(\nabla_{\lower 2pt\hbox{$\scriptstyle\pi_k X$}}\alpha)(\pi_k Y,\pi_jZ,\pi_jW)=-\big(\alpha(\pi_kY,(\nabla_{\lower 2pt\hbox{$\scriptstyle\pi_k X$}}\pi_j)Z,\pi_jW)+\alpha(\pi_kY,\pi_jZ,(\nabla_{\lower 2pt\hbox{$\scriptstyle\pi_k X$}}\pi_j)W)\big).
\end{multline*}
Changing the roles of the pairs $(\pi_k X,\pi_k Y)$ and $(\pi_j Z,\pi_j W)$ we get
\begin{multline*}
(\nabla_{\lower 2pt\hbox{$\scriptstyle\pi_j Z$}}\alpha)(\pi_j W,\pi_kX,\pi_kY)=-\big(\alpha(\pi_jW,(\nabla_{\lower 2pt\hbox{$\scriptstyle\pi_j Z$}}\pi_k)X,\pi_kY)+\alpha(\pi_jW,\pi_kX,(\nabla_{\lower 2pt\hbox{$\scriptstyle\pi_j Z$}}\pi_k)Y)\big).
\end{multline*}
Taking the difference of the two preceding identities
the Codazzi equation from Proposition\,\ref{cod} completes the proof.
\end{proof}

In the next theorem we finally provide the invariant Weyl Identity.

\begin{theorem}
\label{weylt}
We have the identity
\begin{align*}
-4i&g^2\,T(X,Y,Z,W)
\\=\sum_{\ell,\,j=0}^{g-1}&tr_{\hat{g}}\Big(
\alpha(B_0^{-\ell}\,Y,B_0^{-j}\,W,\,\cdot\,)\,\sum_{k=0}^{j-1}\big( \alpha(B_0^{\ell}\,X,B_0^{k+1}\,Z,B_0^{j-k-1}\,\cdot\,)+\alpha(B_0^{\ell}\,X,B_0^{k}\,Z,B_0^{j-k}\,\cdot\,)\big)\\
&+\alpha(B_0^{-\ell}\,Y,B_0^{-j}\,Z,\,\cdot\,)\,\sum_{k=0}^{j-1}\big( \alpha(B_0^{\ell}\,X,B_0^{k+1}\,W,B_0^{j-k-1}\,\cdot\,)+\alpha(B_0^{\ell}\,X,B_0^{k}\,W,B_0^{j-k}\,\cdot\,)\big)\\
&-\alpha(B_0^{-j}\,W,B_0^{-\ell}\,Y,\,\cdot\,)\,\sum_{k=0}^{\ell-1}\big( \alpha(B_0^j\,Z,B_0^{k+1}\,X,B_0^{\ell-k-1}\,\cdot\,)+\alpha(B_0^j\,Z,B_0^{k}\,X,B_0^{\ell-k}\,\cdot\,)\big)\\
&-\alpha(B_0^{-j}\,W,B_0^{-\ell}\,X,\,\cdot\,)\,\sum_{k=0}^{\ell-1}\big( \alpha(B_0^j\,Z,B_0^{k+1}\,Y,B_0^{\ell-k-1}\,\cdot\,)+\alpha(B_0^j\,Z,B_0^{k}\,Y,B_0^{\ell-k}\,\cdot\,)\big)\Big).
\end{align*}
\end{theorem}
\begin{proof}
Take the sum from $1$ to $g$ over $j$ and $k$ of the identities just proved in Lemma\,\ref{pwe},
and use the identity
$$\sum_{\gamma=1}^g\pi_{\gamma}\otimes\pi_{\gamma}=\tfrac{1}{g}\sum_{\gamma=1}^gB_0^{\gamma}\otimes B_0^{-\gamma}.$$
The claim then follows from Lemma\,\ref{ableitungb}.
\end{proof}

In terms of the better invariants $(\hat{g},\alpha,B\otimes B^{-1})$, the classical Weyl Identities thus condense into one structural tensor identity. 
This, in particular, makes it feasible to consider higher derivatives of the Weyl Identity.
In the classical approaches this is not possible without considering a plethora of different cases.

\subsubsection{The importance of the Weyl Identity}
\label{weyl}
In this subsection we explain the importance of the Weyl Identity.
First we relate the Weyl Identity to several well-known identities, e.g. the Cartan identity.
Afterwards we explain why an invariant formulation of the Weyl Identities is important.

\smallskip

Although in most parts of the literature the Weyl Identity does not occur explicitly, 
it however plays a decisive role in all papers concerned with the classification of isoparametric hypersurfaces in spheres. 
Karcher was the first to prove the classical Weyl Identities, in fact, \cite{karcher} is the only source mentioning them explicitly.
Karcher showed that for $g=3$ the Weyl Identities turn each curvature distribution $D_j$ into a normed algebra and thus reproved the results of Cartan in a structural way.
 
 \smallskip

 {\bf Relation to the Cartan Identity.} Our first observation is that the Weyl Identities imply the Cartan identity. Before proving this, we
 recall the Cartan identity
\begin{align*}
\sum_{\substack{j=1 \\ j\ne i}}^g m_j\,\frac{1+\lambda_i\lambda_j}{\lambda_i-\lambda_j}=0,\hspace{1cm}i\in\left\{1,...,g\right\},
\end{align*}
where we make use of the short hand notation $\lambda_i=\lambda_i(0)$.
This identity is crucial in Cartan's work on isoparametric hypersurfaces is space forms \cite{cartan1}-\cite{cartan4}.
Using this identity, Cartan classified isoparametric hypersurfaces of Euclidean spaces and hyperbolic spaces.
Cartan in particular proved that for these cases the number $g$ of distinct principal curvatures is at most two.
However, for the case where the ambient space is a sphere, this identity does not provide such strong restrictions on $g$.

\smallskip

Nomizu \cite{nomizu} proved that the Cartan identity is equivalent to the minimality of the focal submanifolds.
Indeed, by (\ref{specfoc}) we obtain
\begin{align*}
\tr(A_{\theta_i\vert \nu_p})=\sum_{\substack{j=1 \\ j\ne i}}^gm_j\cot(\theta_j-\theta_i)=\sum_{\substack{j=1 \\ j\ne i}}^gm_j\frac{1+\lambda_i\lambda_j}{\lambda_i-\lambda_j}.
\end{align*}

We now prove that the Weyl Identities imply the Cartan identity.

\begin{lemma}
\label{weylcartan}
The Weyl Identities imply the Cartan identity.
\end{lemma}
\begin{proof}
Denote by $(f_k)_{k=1}^n$ an $g_0$-orthonormal frame of $TM$ which consists of eigenvector fields of $A_0$.
Choosing $v_i=\widetilde{v_i}=f_i$ and $v_j=\widetilde{v_j}=f_j$ in Corollary \ref{wij2}, we get
\begin{align*}
1+\lambda_i\lambda_j=2\sum_{k=1,\lambda_k\neq\lambda_j,\lambda_i}^n\frac{\alpha(f_k,f_i,f_j)^2}{(\lambda_i-\lambda_k)(\lambda_j-\lambda_k)}.
\end{align*}
Hence we obtain
\begin{align*}
\sum_{j=1,\lambda_j\neq\lambda_i}^n\frac{1+\lambda_i\lambda_j}{\lambda_i-\lambda_j}&=2\,\hat{\sum}_{k,j}\frac{\alpha(f_i,f_j,f_k)^2}{(\lambda_i-\lambda_j)(\lambda_j-\lambda_k)(\lambda_i-\lambda_k)}\\&=-
\sum_{k=1,\lambda_k\neq\lambda_i}^n\frac{1+\lambda_i\lambda_k}{\lambda_i-\lambda_k}
\end{align*}
where we denote by $\hat{\sum}_{k,j}$ the sum over those $j,k\in\left\{1,...,g\right\}$ with $\lambda_k\neq\lambda_j\neq\lambda_i\neq\lambda_k.$
Consequently, we get
\begin{align*}
\sum_{j=1, j\neq i}^gm_j\frac{1+\lambda_i\lambda_j}{\lambda_i-\lambda_j}=
\sum_{j=1,\lambda_j\neq\lambda_i}^n\frac{1+\lambda_i\lambda_j}{\lambda_i-\lambda_j}=0,
\end{align*}
i.e. the Cartan identity.
\auxqed
\end{proof}
Clearly, the Cartan identity is weaker than the Weyl Identity.

\smallskip

{\bf Relation to the isospectral families $L(s)$.}
The classification of the isospectral families $L(s)$ at one focal manifold is the central step for the classification of isoparametric hypersurfaces in spheres with $(g,m)=(6,1)$ - also see Subsection \ref{g=6}.
Next we show that the Weyl Identities encode the isospectrality of $L(s)$.

\smallskip

Let $\overline{p}\in M_j$.
It is well-known, see e.g. \cite{mi2}, that $T_{\overline{p}}M_j$ maybe identified with $\oplus_{i\neq j}D_i(q)$ for any $q\in F_{\theta_j}^{-1}(\overline{p})$.
Consequently, the normal space $\nu_{\overline{p}}M_j$ of $M_j$ at $\overline{p}$ is spanned by $\nu_{\theta_j}(p)$ and a basis of $D_j(p)$.
Recall that for each choice of pairs of orthogonal vectors $\nu_1,\nu_2\in\nu_{\overline{p}}M_j$  one gets a isospectral family $L(s)=\cos(s)A_{\nu_1}+\sin(s)A_{\nu_2}$.
We observe that the condition that $L(s)$ is isospectral partially encodes the Weyl Identity and higher covariant derivatives thereof.
In the following theorem we shall make this statement more precise in the case $(g,m)=(6,1)$ only.

\begin{theorem}
\label{l1}
Let $(g,m)=(6,1)$ and $e_i\in D_i$ be unit vector fields.
Furthermore, let $\overline{p}\in M_6$ and $p\in F_{\theta_j}^{-1}(\overline{p})$.
Denote by $L_0$ and $L_1$ the shape operator of $M_6$ at $\overline{p}\in M_6$ with respect to $\nu_{\theta_6}(p)$ and $e_6(p)$, respectively.
The isospectrality of  $L(s)=\cos(s)L_0+\sin(s)L_1$ is equivalent to the classical Weyl Identity with $(i,j)=(3,6)$ and the first four covariant derivatives with respect to $e_6\in D_6$ thereof.
\end{theorem}
\begin{proof}
We only give a sketch of the proof.
First we verify
 $L_0=\mbox{Diag}(\sqrt{3},\frac{1}{\sqrt{3}},0,-\frac{1}{\sqrt{3}},-\sqrt{3})$ and
\begin{align}
\label{l1inalpha}
L_1=\left(\begin{array}{ccccc}
0&\sqrt{\frac{2}{3}}\,\alpha_{1\,2\,6}&\frac{1}{\sqrt{2}}\,\alpha_{1\,3\,6}&\sqrt{\frac{2}{3}}\,\alpha_{1\,4\,6}&\sqrt{2}\,\alpha_{1\,5\,6}\\
\sqrt{\frac{2}{3}}\,\alpha_{1\,2\,6}&0&\frac{1}{\sqrt{6}}\,\alpha_{2\,3\,6}&\frac{\sqrt{2}}{3}\,\alpha_{2\,4\,6}&\sqrt{\frac{2}{3}}\,\alpha_{2\,5\,6}\\
\frac{1}{\sqrt{2}}\,\alpha_{1\,3\,6}&\frac{1}{\sqrt{6}}\,\alpha_{2\,3\,6}&0&\frac{1}{\sqrt{6}}\,\alpha_{3\,4\,6}&\frac{1}{\sqrt{2}}\,\alpha_{3\,5\,6}\\
\sqrt{\frac{2}{3}}\,\alpha_{1\,4\,6}&\frac{\sqrt{2}}{3}\,\alpha_{2\,4\,6}&\frac{1}{\sqrt{6}}\,\alpha_{3\,4\,6}&0&\sqrt{\frac{2}{3}}\,\alpha_{4\,5\,6}\\
\sqrt{2}\,\alpha_{1\,5\,6}&\sqrt{\frac{2}{3}}\,\alpha_{2\,5\,6}&\frac{1}{\sqrt{2}}\,\alpha_{3\,5\,6}&\sqrt{\frac{2}{3}}\,\alpha_{4\,5\,6}&0
\end{array}\right),
\end{align}
where $\alpha_{i\,j\,k}=\alpha_{\lvert p}(e_i,e_j,e_k)$ for a $p\in F_{\theta_j}^{-1}(\overline{p})$.
Substitute these results into the minimal polynomial equation for $L(s)$. 
A tedious but straightforward calculation shows that
the ideal generated by the resulting equations coincides with the ideal generated by the classical Weyl Identity with $(i,j)=(3,6)$ and the first four covariant derivatives with respect to $e_6\in D_6$ thereof.
\end{proof}

 {\bf Advantages of the invariant Weyl Identity.}
We shall now describe the advantages of the invariant Weyl Identity deduced in the previous paragraph compared to the classical Weyl Identities.

\smallskip

The discussion above, in particular Theorem\,\ref{l1}, highlights what important role the higher covariant derivatives of the Weyl
Identities play in the classification. 
 Since the classical Weyl Identities depend on several indices, i.e. $i$ and $j$, taking higher covariant derivatives of these identities would lead to a plethora of different cases. By contrast, in terms of the invariants $\hat g$, $\alpha$ and $B\otimes B^{-1}$
it is entirely possible to consider higher covariant derivatives since the Weyl Identities are
condensed in a single tensor identity.

\smallskip

In order to prove homogeneity of isoparametric hypersurfaces with $g=6$, one needs to analyze the interaction of the isospectral families that show up at different focal submanifolds (on the same normal great circle of $M^n$) - see also Subsection\,\ref{g=6} for more details. 
The invariant Weyl Identity contains all the information of isospectral families at different focal submanifolds!

\smallskip

For $g=3$ the Weyl Identities turn each curvature distribution $D_j$ into a normed algebra \cite{karcher}.
An interesting question is if, as for $g=3$, there exists
a geometric structure for $g=6$ captured by the Weyl Identities. 
The existing examples suggest a geometry closely related to $\Gtwo$.
This approach might lead to a viable strategy for completing classification.

Another open question is, whether for $g=4$ the Weyl Identity reflects parts of the Clifford algebra structure, which is the central underlying structure in this case.

\subsection{Symmetry identities}
\label{symm}
Throughout this section $p$ shall denote a fixed point of the manifold $M$.

\smallskip

Let $t\in\R$ and $k\in\N$ be given. The parallel surface map given by 
$$F_t(p)\mapsto F_{t+2(\theta_k-t)}(p)=F_{2\theta_k-t}(p)$$ maps 
the submanifold $F_t(M)\subset\Sph^{n+1}$ onto itself and flips the sign of $\nu_t$. Hence
 there exist diffeomorphisms $\tau_k:M\rightarrow M$ such that
\begin{align*}
F_t\circ\tau_k=F_{2\theta_k-t} \hspace{1cm}\mbox{and}\hspace{1cm} \nu_{t\vert\tau_k(p)}=-\nu_{2\theta_k-t\lvert p}\hspace{0.4cm}\forall p\in M.
\end{align*}
Clearly, the maps $\tau_k:M\rightarrow M$ are reflections in the focal submanifolds, and in particular involutions.

\smallskip

\begin{lemma}
For $j\in\left\{1,...,g\right\}$, the map
$\tau_j:M\rightarrow M$ is an isometry of $(M,\hat{g})$. 
Furthermore, the differences $\theta_j-\theta_k$ generate a discrete cyclic subgroup in $\mathbb{R}/{\mathbb{Z}\pi}$ and the involutions $\tau_j,  1\leq j\leq g$, are the reflections in the dihedral group $D_g=\left\langle \tau_1,\tau_g\right\rangle\subset\mbox{Diff}(M)$.
\end{lemma}

\begin{proof}
The very definition of $\tau_k$ immediately implies
$e^{-2i\theta_j}\hat{F}_0(p)=\hat{F}_{2\theta_j}(p)=\overline{\hat{F}_0(\tau_j(p))}.$
Consequently,
$e^{-2i\theta_j}d\hat{F}_{0\vert p}Y=\overline{d\hat{F}_{0\vert\tau_j(p)}d\tau_{j\vert p}Y}.$ 
Thus we get
\begin{align*}
\hat{g}_{\vert p}(X,Y)=\hat{g}_{\vert \tau_j(p)}(d\tau_{j\vert p}X,d\tau_{j\vert p}Y).
\end{align*}
\auxqed
\end{proof}

In the next theorem we prove the identities which we call \textit{Symmetry identities}.
\begin{theorem}
\label{glob}
For $j\in\left\{1,...,g\right\}$ and $p\in M$
\begin{align*}
\alpha_{\vert p}(X,Y,Z)=-\alpha_{\vert \tau_j(p)}(d\tau_{j\vert p}X,d\tau_{j\vert p}Y,d\tau_{j\vert p}Z),
\end{align*}
or, for short, $(\tau_j)_*\alpha = -\alpha$.
Furthermore, the higher covariant derivatives of $\alpha$ transform 
exactly as $\alpha$ does under $\tau_j$, i.e.,
 $\nabla^i\alpha=-(\tau_j)_{*}\nabla^i\alpha$ for all $i\geq 0.$
\end{theorem}
\begin{proof}
By Theorem\,\ref{alpha} and $F_0\circ\tau_j=F_{2\theta_j}$ we have
\begin{align*}
\alpha_{\vert p}(X,Y,Z)&=\alpha^{2\theta_j}_{\vert p}(X,Y,Z)
\\&=\left\langle dF_{0\vert \tau_j(p)}d\tau_{j\vert p}X,dF_{0\vert \tau_j(p)}d\tau_{j\vert p} (\nabla_{Y}^{2\theta_j}A_{2\theta_j})Z\right\rangle_{\Sph^{n+1}}\\&=
g_{0\vert\tau_j(p)}(d\tau_{j\vert p}X,d\tau_{j\vert p}(\nabla_{Y}^{2\theta_j}A_{2\theta_j})Z).
\end{align*}
The identities $F_0\circ\tau_j=F_{2\theta_j}$ and $\nu_0\circ\tau_j=-\nu_{2\theta_j}$ imply
\begin{align*}
A_{0\vert\tau_j(p)}d\tau_{j\vert p}X_1=-d\tau_{j\vert p}A_{2\theta_j\vert p}X_1
\end{align*}
for all $p\in M$ and for all $X_1\in T_pM$. Moreover, the first identity also implies that $\tau_j:(M, g_{2\theta_j})\rightarrow (M, g_0)$ is an isometry.
Thus we get
\begin{align*}
&g_{0\vert\tau_j(p)}(d\tau_{j\vert p}X,d\tau_{j\vert p}(\nabla_{Y}^{\theta_j}(A_{2\theta_j}Z)-A_{2\theta_j}\nabla^{2\theta_j}_{Y}Z))
\\&=
-g_{0\vert\tau_j(p)}(d\tau_{j\vert p}X,(\nabla_{d\tau_{j\vert p}Y}^{0}(A_{0}d\tau_{j\vert p}Z)-A_{0}\nabla^{0}_{d\tau_{j\vert p}Y}
d\tau_{j\vert p}Z))
\\&=-\alpha_{\vert \tau_j(p)}(d\tau_{j\vert p}X,d\tau_{j\vert p}Y,d\tau_{j\vert p}Z),
\end{align*}
and thus the first claim. From this the second claim is immediate. 
\end{proof} 

\begin{remark}
\label{global}
\begin{enumerate}
\item Note that the Symmetry Identities relate $\alpha_p$ to $\alpha_q$, where $p$ and $q$ are different points of $M$.
This means that in contrast to the Weyl Identities, the Symmetry identities are not pointwise identities.
\item In Section\,4, Lemma 4.1 in \cite{mi2} Miyaoka states some identities, which she refers to as \lq global symmetry\rq, and which were deduced by her in \cite{mi5}. 
These identities are presumably equivalent to the Symmetry Identities.
However, the author does not understand the proof of the \lq global identities\rq\ in \cite{mi5}.
\end{enumerate}
\end{remark}

\section{A geometric interpretation of homogeneity}
\label{sec4}
In this section we prove that homogeneity of isoparametric hypersurfaces with $g=6$ is equivalent to a geometric property of the Lagrangian submanifold in the complex quadric.
We hope that a detailed study of the geometry of the Lagrangian submanifold finally will lead to a geometric classification of isoparametric hypersurfaces in spheres with $g=6$.

\smallskip

This section is structured as follows:
in the first subsection we recall what is known for the case $g=6$, in the second subsection we determine $\alpha$ for the homogeneous examples with $g=6$.
Finally, in the third subsection, we give several equivalent formulations of homogeneity of isoparametric hypersurfaces with $g=6$.

\subsection{Isoparametric hypersurfaces with $g=6$}
\label{g=6}
In this subsection we summarize the known results for isoparametric hypersurfaces in spheres with $g=6$.
 
\smallskip

For the case of isoparametric hypersurfaces in spheres with $g=6$, all multiplicities coincide and are given either by $m=1$ or $m=2$ \cite{abresch}. 
Furthermore, exactly two examples with $g=6$ are known, both of which are homogeneous.
They are given as orbits of the isotropy representation of $\Gtwo/\SO(4)$ or as orbits in the unit sphere $\Sph^{13}$ of the Lie
algebra $\frak{g}_2$ of the adjoint representation of the Lie group $\Gtwo$ and have multiplicities $m=1$ and $m=2$, respectively.
The following conjecture is due to Dorfmeister and Neher and is believed to be true.

\medskip

\noindent{\bf Conjecture (\cite{dn}):} Each maximal isoparametric hypersurface with $g=6$ principal curvatures is homogeneous.

\medskip

Dorfmeister and Neher proved this conjecture in the affirmative for the case $m=1$. 
Since homogeneous isoparametric hypersurfaces in spheres were classified by Takagi and Takahashi \cite{tt}, this provides a classification of 
isoparametric hypersurfaces with $(g,m)=(6,1)$. 
Similarly, proving that isoparametric hypersurfaces with $(g,m)=(6,2)$ are homogeneous, would yield a classification of such hypersurfaces. 
Note that the case $m=2$ is not classified yet, see the appendix of this paper. 
Therefore proving the above conjecture still remains the goal for isoparametric hypersurfaces with $(g,m)=(6,2)$.

\smallskip

Below we explain the approach by Dorfmeister and Neher for the case $m=1$. 
The starting point of their work is the following algebraic description of isoparametric hypersurfaces in spheres which is due to M\"unzner.

\begin{theorem}[\cite{munzner}] 
a) Let $M\subset\Sph^{n+1}$ be an isoparametric hypersurface with $g$ distinct eigenvalues.
Then there exists a homogeneous polynomial $F:\R^{n+2}\rightarrow\R$
of degree $g$ and positive integers $m_1$ and $m_2$ such that\\ 
$M$ is an open submanifold of a level surface $$M_t=\Sph^{n+1}\cap F^{-1}(t)$$ \quad for a $t\in(-1,1)$,
and the identities $$\langle\mbox{grad}F(x),\mbox{grad}F(x)\rangle=g^2\langle x,x\rangle^{g-1},$$
$$\Delta F(x)=\tfrac{1}{2}(m_2-m_1)g^2\langle x,x\rangle^{g/2-1}$$ and
$n=\tfrac{g}{2}(m_1+m_2)$ are satisfied.\\
b) Conversely, for each homogeneous polynomial $F$ of degree $g$ satisfying the three identities in a), the level surfaces $M_t$, $t\in(-1,1)$, are 
isoparametric.
\end{theorem}

For the case of isoparametric hypersurfaces with $(g,m)=(6,1)$, Dorfmeister and Neher proved that there exists - up to isomorphism - 
only one isoparametric polynomial in $\R^8$.
The central step in their proof is a partial classification of the so-called $E$-families.
Dorfmeister and Neher provided homogeneity by showing that only one of the explicit examples 
of $E$-families is associated to an isoparametric hypersurface in a sphere.

\smallskip

In \cite{siffert} the author gave a simplified proof of the theorem of Dorfmeister and Neher.
The central step in \cite{siffert} consists in classifying the isospectral families at one focal submanifold, which can be shown to be equivalent to classifying the $E$-families introduced in \cite{dn}.
Below we reformulate the essential insights from \cite{dn} and \cite{siffert} in terms of the isospectral families, since 
we use this notation throughout this paper.

\smallskip

The homogeneity of isoparametric hypersurfaces with $g=6$ is equivalent to the property that the kernels of the isospectral families $L(s)$ are independent of $s$ \cite{dn, mi2}. 
Although requiring the family $L(s)$ have eigenvalues $\pm\sqrt{3}, \pm1/\sqrt{3}$ and $0$, all with the same multiplicity $m$, is a very restrictive condition on the symmetric real $5m\times 5m$-matrices $A_{\nu_1}$ and $A_{\nu_2}$, so far no one has yet succeeded in classifying such matrices for $m\ge 2$.
Examples of such matrices are provided by the irreducible representations of $\SU(2)$.
Among these examples one finds cases where the kernel of $L(s)$ is not constant when varying $s$.
To prove homogeneity of isoparametric hypersurfaces with $g=6$, it thus does not suffice to study the properties of just one isospectral family.
One also needs to analyze the interaction of the isospectral families that show up at different focal submanifolds (on the same normal great circle of $M$). 

\smallskip

Miyaoka also worked on the classification of isoparametric hypersurfaces with $g=6$.\\
Her work on the case $(g,m)=(6,1)$ is contained in \cite{mi1} and the corresponding Erratum \cite{mi1err}.
However, there is still a crucial gap in the Erratum \cite{mi1err} - see \cite{siffert} for details.\\
Miyaoka's work on the case $(g,m)=(6,2)$ is contained in \cite{mi2} and the corresponding Erratum \cite{mierr}.
However, there is also a crucial gap in the Erratum \cite{mierr} - see the appendix of the present paper.

\subsection{Calculation of $\alpha$ for the homogeneous examples with $g=6$}
\label{chhom}
In the case $g = 6$ only two examples are known, both of which are homogeneous. They are given as orbits of the
isotropy representation of $\Gtwo/\SO(4)$ or the compact real Lie group $\Gtwo$, respectively.
In both cases all six principal curvatures coincide and are given by $m=1$ and $m=2$, respectively.

\smallskip

For both of the homogeneous examples Miyaoka \cite{mi3,mi4} calculated the Christoffel 
symbols $$\Lambda_{i,j}^k:=g_{0}(\nabla^0_{f_i}f_j,f_k),$$ where $(f_n)_{n=1}^{6m}$ is 
a $g_0$-orthonormal frame with $f_{i+6k}\in D_{i}$ for $i\in\{1,\cdots, 6\}$ and $k=0,\cdots,m-1$.
In what follows we use these results to determine $\alpha$ for the homogeneous examples.

\smallskip

From $A_0f_i=\lambda_{i}(0)f_i$ we obtain
\begin{align*}
(\nabla^0_{X}A_0)f_i=(\lambda_{i}(0)-A_0)\nabla^0_{X}f_i,
\end{align*}
where $X\in\Gamma(TM)$ and the index $i$ in $\lambda_i(0)$ is interpreted to be cyclic of order $6$.
Thus for $j\neq k$ we get 
\begin{align*}
\Lambda_{i,j}^k=(\lambda_{j}(0)-\lambda_{k}(0))^{-1}\alpha(f_i,f_j,f_k).
\end{align*}
Instead of calculating $\alpha(f_i,f_j,f_k)$ we determine $\alpha(e_i,e_j,e_k)$, where $(e_i)_{i=1}^{6m}$ denotes the $\hat{g}$-orthonormal basis with $e_i\in D_{i}$, which is associated to the $g_0$-orthonormal basis $(f_i)_{i=1}^{6m}$. In other words,
\begin{align*}
e_i=\sqrt{2(1+(\lambda_{i}(0))^2)^{-1}}f_i,
\end{align*}
for $i\in\left\{1,...,6m\right\}$.

\smallskip

Substituting the Christoffel symbols \cite{mi3}  into the above equation we obtain the following lemma.

\begin{lemma}
\label{heee}
For the homogeneous isoparametric hypersurfaces with $(g,m)=(6,1)$ the components $\alpha_{i\,j\,k}:=\alpha(e_i,e_j,e_k)$ are given by
\begin{align*}
\alpha_{1\,2\,3}=\alpha_{3\,4\,5}=\alpha_{1\,5\,6}=\sqrt{\tfrac{3}{2}},\,\alpha_{2\,4\,6}=-\sqrt{\tfrac{3}{2}},\,\alpha_{1\,3\,5}=-2\sqrt{\tfrac{3}{2}}.
\end{align*}
All other $\alpha_{i\,j\,k}$ with $i\leq j\leq k$ vanish.
\end{lemma}

Next we consider the case $(g,m)=(6,2).$ Following Miyaoka, we use the notation $\overline{f}_i:=f_{6+i},i\in\left\{1,...,6\right\}$.
Furthermore, an entry $\overline{e}_i$ of $\alpha$ will be denoted by an index $\overline{i}$, e.g., $\alpha(\overline{e}_1,e_5,e_6)$ is denoted by $\alpha_{\overline{1}\,5\,6}.$
Clearly, $f_i$ and $\overline{f}_i$ constitute an orthonormal basis of the two-dimensional distribution $D_i$.

\begin{remark}
 Note, that the above choice of $f_i$ and $\overline{f}_i$ is not canonical.
This freedom in the choice of the basis, is one of the reasons why computer computations, which aim to determine the possible $\alpha$, fail until today.
 \end{remark}
 
Substituting the Christoffel symbols \cite{mi4} into the above equation we obtain the following lemma.
\begin{lemma}
\label{heee2}
For the homogeneous isoparametric hypersurfaces with $(g,m)=(6,2)$ the components $\alpha_{i\,j\,k}:=\alpha(e_i,e_j,e_k)$ are given by
\begin{align*}
&\alpha_{1\,\overline{5}\,6}=-\sqrt{\tfrac{3}{2}},\,\,\alpha_{\overline{1}\,5\,6}=\alpha_{1\,5\,\overline{6}}=\alpha_{\overline{1}\,\overline{5}\,\overline{6}}=\sqrt{\tfrac{3}{2}},\,\,\,\alpha_{2\,\overline{4}\,6}=-\sqrt{\tfrac{3}{2}},\,\,\alpha_{\overline{2}\,4\,6}=\alpha_{2\,4\,\overline{6}}=\alpha_{\overline{2}\,\overline{4}\,\overline{6}}=\sqrt{\tfrac{3}{2}},\\
&\alpha_{1\,\overline{2}\,3}=-\sqrt{\tfrac{3}{2}},\,\,\alpha_{\overline{1}\,2\,3}=\alpha_{1\,2\,\overline{3}}=\alpha_{\overline{1}\,\overline{2}\,\overline{3}}=\sqrt{\tfrac{3}{2}},\,\,\,\alpha_{3\,\overline{4}\,5}=\sqrt{\tfrac{3}{2}},\,\,\alpha_{\overline{3}\,4\,5}=\alpha_{3\,4\,\overline{5}}=\alpha_{\overline{3}\,\overline{4}\,\overline{5}}=-\sqrt{\tfrac{3}{2}},\\
&\alpha_{1\,\overline{3}\,5}=2\sqrt{\tfrac{3}{2}},\,\,\alpha_{\overline{1}\,3\,5}=\alpha_{1\,3\,\overline{5}}=\alpha_{\overline{1},\overline{3},\overline{5}}=-2\sqrt{\tfrac{3}{2}}.
\end{align*}
All other $\alpha_{i\,j\,k}$ with $i\leq j\leq k$ vanish.
\end{lemma}

We used the above results to guess equivalent formulations for homogeneity. 
The following subsection contains our results of this procedure.

\subsection{Equivalent formulations of homogeneity}
In this subsection we provide several equivalent formulations of homogeneity.
Throughout this subsection let $X,Y,Z\in\Gamma(TM)$ and $i\in\left\{1,2,3\right\}$.

\smallskip

For proving an extended version of Theorem\,B we need two preparatory lemmas.

\begin{lemma}
\label{curvten}
Let $g=6$. For each $i\in\left\{1,...,6\right\}$ the identity
\begin{align*}
R(\pi_iX,\pi_{i+3}Y,\pi_{i+3}Y,\pi_iX)=\tfrac{1}{4}\tr_{\hat{g}}(\alpha(\pi_iX,\pi_{i+3}Y,\,\cdot\,)^2)
\end{align*}
holds, the index of the projections is interpreted to be cyclic of order $6$.
\end{lemma}
\begin{proof}
An easy calculation yields $(\hat{g}\owedge \hat{g}+b\owedge\overline{b})(\pi_iX,\pi_{i+3}Y,\pi_{i+3}Y,\pi_iX)=0$
and thus the claim follows from Proposition\,\ref{cod}.
\end{proof}

\begin{lemma}
\label{dist}
Let $j\in\left\{1,...,6\right\}$.
For each vector field $Z\in\Gamma(TM)$ introduce the vector field  $\hat{Z}:=\pi_{\lower 2pt\hbox{$\scriptstyle{D_j\oplus D_{j+3}}$}}Z$.
The direct sum $D_j\oplus D_{j+3}$ is integrable if and only if
\begin{align}
\label{inte}
(\nabla_{\hat{X}}B_{\theta_j}^2)\hat{Y}-(\nabla_{\hat{Y}}B_{\theta_j}^2)\hat{X}=0
\end{align}
holds for all $X,Y\in\Gamma(TM).$
\end{lemma}
\begin{proof}
By Section \ref{spro} we obtain
\begin{align*}
\pi_{\lower 2pt\hbox{$\scriptstyle{D_j\oplus D_{j+3}}$}}=\tfrac{1}{3}(\eins+B^{2}_{\theta_j}+B^{4}_{\theta_j}),
\end{align*}
for $j\in\left\{1,...,6\right\}$. 
Hence, using $B_{\theta_j}^6=\eins$ we get the identity $(B_{\theta_j}^2-\eins)\hat{Y}=0$
and thus
\begin{align*}
(\nabla_{\hat{X}}B_{\theta_j}^2)\hat{Y}=(\eins-B_{\theta_j}^2)\nabla_{\hat{X}}\hat{Y}.
\end{align*}
By interchanging the roles of $X$ and $Y$ and subtracting the resulting equation from the preceding equation we obtain 
\begin{align*}
(\nabla_{\hat{X}}B_{\theta_j}^2)\hat{Y}-(\nabla_{\hat{Y}}B_{\theta_j}^2)\hat{X}=(\eins-B_{\theta_j}^2)\big(\nabla_{\hat{X}}\hat{Y}-\nabla_{\hat{Y}}\hat{X}\big)=(\eins-B_{\theta_j}^2)\left[\hat{X},\hat{Y}\right].
\end{align*}
By definition of $B_t$ we get $(\eins-B_{\theta_j}^2)Z=0$ if and only if $Z\in D_j\oplus D_{j+3}$. Combining this with the previous identity yields the desired result.
\end{proof}

In the next theorem we finally provide several equivalent formulations for homogeneity of isoparametric hypersurfaces in spheres with $g=6$.

\begin{theorem}
\label{theoremzus}
Each of the following statements is equivalent to the homogeneity of isoparametric hypersurfaces in spheres with $g=6$.
\begin{enumerate}
\renewcommand{\labelenumi}{(\roman{enumi})}
\item For each $i\in\left\{1,...,6\right\}$ we have
\begin{align*}
\alpha(\pi_iX,\pi_{i+3}Y,Z)=0.
\end{align*}	
\item\label{1} For $i,j,k\in\lbrace 1,...,6\rbrace$ with $j+k+l\not\equiv 0\,\,\mbox{modulo}\,\,3$ we have
\begin{align*}
\alpha(\pi_jX,\pi_{k}Y,\pi_{l}Z)=0.
\end{align*}
\item The following identity is satisfied
\begin{align*}
\sum_{j=0}^5{(-1)^j\alpha(B^jX,B^{-j}Y,Z)}=0.
\end{align*}
\item The following sectional curvatures 
of $(M,\hat{g})$ vanish:
\begin{align*}
R(\pi_iX,\pi_{i+3}Y,\pi_{i+3}Y,\pi_iX),\,i\in\left\{1,...,6\right\}.
\end{align*}
\item For $j\in\left\{1,...,6\right\}$ the direct sum $D_j\oplus D_{j+3}$ is integrable.
\item The kernel of each 
linear isospectral family $L(s)$ is independent of $s\in\R$. 
\end{enumerate}
\end{theorem}

\begin{proof}
It is well-known that the sixth statement is equivalent to the homogeneity of isoparametric hypersurfaces in spheres with $g=6$ \cite{dn, mi2}.
Hence it is sufficient to prove that the six statements are equivalent to each other.
\begin{itemize}
\item $(i)\Rightarrow (ii)$. 
Lemma\,\ref{van} implies that 
$\alpha(\pi_jX,\pi_{k}Y,\pi_{l}Z)\neq 0$
can only hold if $(j,k,l)=(n,n+2,n+4)$ or $(j,k,l)=(n,n+1,n+2)$, up to a permutation of $n\in\{1,\ldots,6\}$. 
Since in these cases the equation $i+j+k=0$ holds modulo $3$, the claim is proved.
\item $(ii)\Rightarrow (i)$. Choose $j=i$ and $k=i+3$ for some $i\in\left\{1,...,6\right\}.$
Then
$$\alpha(\pi_iX,\pi_{i+3}Y,\pi_{l}Z)=0$$
unless $l\in\left\{i,i+3\right\}$.
For $l\in\left\{i,i+3\right\}$ the vanishing follows from Lemma\,\ref{van}.
\item $(i)\Rightarrow (iii)$. 
Using Lemma\,\ref{pro} we get
\begin{align*}
\sum_{\ell=1}^6{\pi_{\ell}\otimes\pi_{\ell+3}}=\tfrac{1}{6^2}\sum_{\ell=0}^5{\sum_{k,j=0}^5{B_{\theta_{\ell}}^k\otimes B_{\theta_{\ell+3}}^j}}=\tfrac{1}{6^2}\sum_{\ell=0}^5{\sum_{k,j=0}^5{(-1)^jB_{\theta_{\ell}}^k\otimes B_{\theta_{\ell}}^j}},
\end{align*}
where we made use of the last identity of Lemma\,\ref{bthe} and $\theta_{\ell+3}=\theta_{\ell}+\frac{\pi}{2}$ to obtain the last equality.
By Lemma\,\ref{bthe} again we obtain
\begin{align*}
\tfrac{1}{6^2}\sum_{\ell=0}^5{\sum_{k,j=0}^5{(-1)^jB_{\theta_{\ell}}^k\otimes B_{\theta_{\ell}}^j}}=
\tfrac{1}{6^2}\sum_{\ell=0}^5{\sum_{k,j=0}^5{(-1)^j\xi^{(j+k)\ell}B_{-\pi/12}^k \otimes B_{-\pi/12}^j}},
\end{align*}
where we introduced $\xi=e^{-i\frac{\pi}{3}}$.
Thus we get
\begin{align*}
\tfrac{1}{6^2}\sum_{\ell=0}^5{\sum_{k,j=0}^5{(-1)^j\xi^{(j+k)\ell}B_{-\pi/12}^k \otimes B_{-\pi/12}^j}}=
\tfrac{1}{6}\sum_{j=0}^5{(-1)^jB_0^j \otimes B_0^{-j}},
\end{align*}
where we made use of Lemma\,\ref{bthe} to get the last equality.
Combined we get
\begin{align*}
\sum_{\ell=1}^6{\pi_{\ell}\otimes\pi_{\ell+3}}=\tfrac{1}{6}\sum_{j=0}^5{(-1)^jB_0^j \otimes B_0^{-j}}.
\end{align*}
Hence $(i)$ implies
\begin{align}
\label{summe}
\tfrac{1}{6}\sum_{j=0}^5{(-1)^j\alpha(B_0^jX,B_0^{-j}Y,Z)}=\sum_{i=1}^6{\alpha(\pi_i^{t}X,\pi_{i+3}^{t}Y,Z)}=0.
\end{align}

\item $(iii)\Rightarrow (i)$. Let $1\leq k\leq 6$ be given. Substitute  
$X=\pi_kX_1$ and $Y=\pi_{k+3}Y_1$, $X_1,Y_1\in\Gamma(TM)$, in equation (\ref{summe}) and use $\pi_i\circ\pi_j=\delta_{i,j}$
for any $i,j\in\lbrace1,..,6\rbrace$.

\item $(i)\Leftrightarrow (iv)$. 
 One direction is immediate from Lemma\,\ref{curvten} and the other follows from the fact that $\alpha$ is real.
\item $(i)\Rightarrow (v)$. Equation (\ref{inte}) is equivalent to the statement $$\hat{g}((\nabla_{\hat{X}}B_{\theta_j}^2)\hat{Y}-(\nabla_{\hat{Y}}B_{\theta_j}^2)\hat{X},Z)=0$$ for all $X,Y,Z\in\Gamma(TM)$. Making use of Lemma\,\ref{ableitungb} and the equalities $B_{\theta_j}\pi_j X=\pi_j X$ and $B_{\theta_j}\pi_{j+3} X=-\pi_{j+3} X$ one finds that the preceding equation is equivalent to
\begin{align*}
\alpha(\hat{X},B_{\theta_j}\hat{Y},B_{\theta_j}Z)-\alpha(B_{\theta_j}\hat{X},\hat{Y}, B_{\theta_j}Z)=0.
\end{align*}
This equation is satisfied since 
\begin{align*}
\alpha((\pi_i+\pi_{i+3})\otimes(\pi_i+\pi_{i+3})\otimes\eins)=0
\end{align*}
holds by $(ii)$ and Lemma\,\ref{van}.
\item $(v)\Rightarrow (i)$.  
Making again use of the equations $B_{\theta_j}\pi_j X=\pi_j X$ and $B_{\theta_j}\pi_{j+3} X=-\pi_{j+3} X$  one verifies easily that 
equation (\ref{inte}) is equivalent to
\begin{align*}
\alpha(\pi_jX,\pi_{j+3}Y,Z)-\alpha(\pi_{j+3}X,\pi_{j}Y,Z)=0\hspace{0.5cm}\mbox{for all}\hspace{0.2cm} Z\in\Gamma(TM).
\end{align*}
Applying this equation to $X=\pi_jX_1$ and $Y=\pi_{j+3}Y_1$, for arbitrary $X_1,Y_1\in\Gamma(TM)$, yields the claim.
\item $(vi)\Rightarrow (i)$. 
We first assume $m=1$.
In the proof of Theorem\,\ref{l1}, we described the linear isospectral family $L(s)=\cos(s)L_0+\sin(s)L_1$, $s\in\R$, of the focal submanifold $F_{\theta_6}$ for the case $m=1$ in terms of $\alpha_{i\,j\,\,k}$, see equation (\ref{l1inalpha}).
Clearly, the kernel of $L_0$ is given by $e_3=(0,0,1,0,0)^{tr}$.
Hence, the constancy of the kernel of $L(s)$ implies that the kernel of $L_1$ is also given by $e_3$.
By (\ref{l1inalpha}), this is equivalent to the identity $\alpha(\pi_3X,\pi_6Y,Z)=0$.
Carrying out analogous considerations for $F_{\theta_j}, j\in\lbrace 1,...,5\rbrace$, we finally get $\alpha(\pi_iX,\pi_{i+3}Y,Z)=0$ for all $i\in\lbrace 1,...,6\rbrace$.\\
The case $m=2$ is proved analogously. Indeed, consider again the linear isospectral family $L(s,t)=\cos(s)L_0+\sin(s)(\cos(t)L_1+\sin(t)L_2)$, $s,t\in\R$, of the focal submanifold $F_{\theta_6}$ in terms of $\alpha_{i\,j\,\,k}$.
Here we have 
\begin{align*}
L_0=\mbox{Diag}(\sqrt{3},\tfrac{1}{\sqrt{3}},0,-\tfrac{1}{\sqrt{3}},-\sqrt{3})\otimes\eins_2,
\end{align*}
\begin{align*}
L_1=\left(\begin{smallmatrix}
0_2&\sqrt{\frac{2}{3}}\,A_{1\,2}&\frac{1}{\sqrt{2}}\,A_{1\,3}&\sqrt{\frac{2}{3}}\,A_{1\,4}&\sqrt{2}\,A_{1\,5}\\
\sqrt{\frac{2}{3}}\,A_{2\,1}&0_2&\frac{1}{\sqrt{6}}\,A_{2\,3}&\frac{\sqrt{2}}{3}\,A_{2\,4}&\sqrt{\frac{2}{3}}\,A_{2\,5}\\
\frac{1}{\sqrt{2}}\,A_{3\,1}&\frac{1}{\sqrt{6}}\,A_{3\,2}&0_2&\frac{1}{\sqrt{6}}\,A_{3\,4}&\frac{1}{\sqrt{2}}\,A_{3\,5}\\
\sqrt{\frac{2}{3}}\,A_{4\,1}&\frac{\sqrt{2}}{3}\,A_{4\,2}&\frac{1}{\sqrt{6}}\,A_{4\,3}&0_2&\sqrt{\frac{2}{3}}\,A_{4\,5}\\
\sqrt{2}\,A_{5\,1}&\sqrt{\frac{2}{3}}\,A_{5\,2}&\frac{1}{\sqrt{2}}\,A_{5\,3}&\sqrt{\frac{2}{3}}\,A_{5\,4}&0_2
\end{smallmatrix} \right),
\end{align*}
\begin{align*}
L_2=\left(\begin{smallmatrix}
0_2&\sqrt{\frac{2}{3}}\,B_{1\,2}&\frac{1}{\sqrt{2}}\,B_{1\,3}&\sqrt{\frac{2}{3}}\,B_{1\,4}&\sqrt{2}\,B_{1\,5}\\
\sqrt{\frac{2}{3}}\,B_{2\,1}&0_2&\frac{1}{\sqrt{6}}\,B_{2\,3}&\frac{\sqrt{2}}{3}\,B_{2\,4}&\sqrt{\frac{2}{3}}\,B_{2\,5}\\
\frac{1}{\sqrt{2}}\,B_{3\,1}&\frac{1}{\sqrt{6}}\,B_{3\,2}&0_2&\frac{1}{\sqrt{6}}\,B_{3\,4}&\frac{1}{\sqrt{2}}\,B_{3\,5}\\
\sqrt{\frac{2}{3}}\,B_{4\,1}&\frac{\sqrt{2}}{3}\,B_{4\,2}&\frac{1}{\sqrt{6}}\,B_{4\,3}&0_2&\sqrt{\frac{2}{3}}\,B_{4\,5}\\
\sqrt{2}\,B_{5\,1}&\sqrt{\frac{2}{3}}\,B_{5\,2}&\frac{1}{\sqrt{2}}\,B_{5\,3}&\sqrt{\frac{2}{3}}\,B_{5\,4}&0_2
\end{smallmatrix} \right),
\end{align*}
where 
\begin{align*}A_{i\,j}=\left(\begin{array}{cc}
\alpha_{i\,j\,6}&\alpha_{i\,\overline{j}\,6}\\
\alpha_{\overline{i}\,j\,6}&\alpha_{\overline{i}\,\overline{j}\,6}\\
\end{array}\right)\quad\mbox{and}\quad B_{i\,j}=\left(\begin{array}{cc}
\alpha_{i\,j\,\overline{6}}&\alpha_{i\,\overline{j}\,\overline{6}}\\
\alpha_{\overline{i}\,j\,\overline{6}}&\alpha_{\overline{i}\,\overline{j}\,\overline{6}}\\
\end{array}\right).
\end{align*}
As in the case $m=1$, the constancy of the kernel of $L(s,t)$ implies the identity $\alpha(\pi_3X,\pi_6Y,Z)=0$.
Again, the claim is established by carrying out analogous considerations for $F_{\theta_j}, j\in\lbrace 1,...,5\rbrace$.
\item $(i)\Rightarrow (vi)$.  Let us first consider the case $m=1$.
Since $\alpha(\pi_iX,\pi_{i+3},Y,Z)=0$ for all $i\in\lbrace 1,...,6\rbrace$, all entries of the third row (and thus the third column) of $L(s)$ are $0$. Thus $e_3$ (the third vector of the standard basis in $\R^5$) lays in the kernel of $L(s)$ for all $s\in\R$.
Since the kernel of $L(s)$ is one dimensional, the claim is established.\\
Next, suppose $m=2$. Since $\alpha(\pi_iX,\pi_{i+3},Y,Z)=0$ for all $i\in\lbrace 1,...,6\rbrace$, all entries of the fifth and the sixth row (and thus the fifth and the sixth column) of $L(s,t)$ are $0$. Thus $e_5$ and $e_6$ (the fifth and sixth vector of the standard basis in $\R^{10}$) lay in the kernel of $L(s,t)$ for all $s,t\in\R$.
Since the kernel of $L(s,t)$ is two dimensional, the claim is established.
\end{itemize}
\auxqed
\end{proof}
Using the classification of isoparametric hypersurfaces in spheres with $g=6$ and $m=1$ given by Dorfmeister and Neher \cite{dn}, we obtain the following corollary.
\begin{corollary}
Assume $(g,m)=(6,1).$ For all $i\in\left\{1,...,6\right\}$ the sectional curvatures
$R(\pi_iX,\pi_{i+3}Y,\pi_{i+3}Y,\pi_iX)=0$
of $(M,\hat{g})$ vanish. 
\end{corollary}

Theorem\,\ref{theoremzus} establishes a new strategy for proving homogeneity of isoparametric
surfaces in spheres with $g=6$:
we hope that a detailed study of the geometry of the Lagrangian submanifold in the complex quadric might 
be helpful.

\section*{Appendix: A counterexample to Miyaoka's proof in \cite{mi2}, \cite{mierr}.}
Let $\overline{p}$ be a point of a fixed focal submanifold $M_+$.
Without loss of generality we assume $M_+=M_6$.
Recall from Subsection \ref{weyl} that $T_{\overline{p}}M_6=\oplus_{i=1}^5D_i(p)$, where $p\in F_{\theta_j}^{-1}(\overline{p})$, and that
the normal space $\nu_{\overline{p}}M_6$ of $M_6$ at $\overline{p}$ is spanned by $\nu_{\theta_6}(p)$ and a basis of $D_j(p)$.

\smallskip

We follow the notation of Miyaoka and let $\eta_p=\nu_{\theta_6}(p)$ and $\xi_p\in D_6(p)$.
In \cite{mi2} Miyaoka introduced
\begin{align*}
E(c)=E(p,\xi_p)=\mbox{span}\{\mbox{Ker}L(t)\,\lvert\,t\in[0,2\pi)\},
\end{align*}
where $c(t)=\cos(t)\eta_p+\sin(t)\xi_p$ and $L(t)=\cos(t)A_{\eta_p}+\sin(t)A_{\xi_p}$ is an isospectral family of focal shape operators.

\smallskip

In \cite{mierr} Miyaoka introduced 
\begin{align*}
E=\mbox{span}\{ E(c)\,\lvert\,\mbox{$c$ geodesic of}\,L_6(p)\},
\end{align*}
where $L_6(p)$ denotes the leaf of $D_6$ through $p$.
Since $m=2$, the unit vector $\xi_p$ is of the form $\xi_p=\cos(s)e_6(p)+\sin(s)e_{\overline{6}}$, where the vectors $e_6(p),e_{\overline{6}}(p)$ constitute an orthonormal basis of $D_6(p)$.
Consequently, we have
$$E=\mbox{span}_{t,s}\mbox{Ker}L(t,s),$$
where $L(t,s)$ is given by
\begin{align*}
L(t,s)=&\cos(t)A_{\eta_p}+\sin(t)A_{\cos(s)e_6(p)+\sin(s)e_{\overline{6}}(p)}\\=&\cos(t)A_{\eta_p}+\sin(t)(\cos(s)A_{e_6(p)}+\sin(s)A_{e_{\overline{6}}(p)}).
\end{align*}

In Proposition 6.2 on page $8$ in the Erratum \cite{mierr}, Miyaoka claims that if $\dim E(c)=4$, then all the shape operators $L(t,s)$ map
$E$ onto $E^{\perp}$.
This statement is not correct, what is shown by the following counterexample.

\begin{counterexample*}
\label{ce}
We give an example of an isoparametric family $L(t,s)$ such that $\dim E(c)=4$ and $\dim E>4$ but
$L(t,s)$ does not map $E$ to $E^{\perp}$.

\smallskip

We use the short hand notation $A_{\eta_p}=L_0$, $A_{e_6(p)}=L_1$ and $A_{e_{\overline{6}}(p)})=L_2$. Furthermore, let 
\begin{align*}
&L_0=\mbox{diag}(\sqrt{3},\sqrt{3},1/\sqrt{3},1/\sqrt{3},0,0,-1/\sqrt{3},-1/\sqrt{3},-\sqrt{3},-\sqrt{3}),\\ 
&L_1=\tfrac{1}{\sqrt{6}}\left(\begin{smallmatrix}
0&0&0&0&3\sqrt{2}\\
0&0&1&0&0\\
0&1&0&1&0\\
0&0&1&0&0\\
3\sqrt{2}&0&0&0&0
\end{smallmatrix} \right)\otimes \eins_2,\\ 
&L_2=\tfrac{1}{\sqrt{6}}\left(\begin{smallmatrix}
0_2&0_2&0_2&0_2&-3\sqrt{2}J\\
0_2&0_2&-J&0_2&0_2\\
0_2&J&0_2&-J&0_2\\
0_2&0_2&J&0_2&0_2\\
3\sqrt{2}J&0_2&0_2&0_2&0_2
\end{smallmatrix} \right)
\end{align*}
where $J=\left(\begin{smallmatrix}
0&-1\\
1&0
\end{smallmatrix} \right)$.
One verifies easily that $L(t,s)=\cos(t)L_0+\sin(t)(\cos(s)L_1+\sin(s)L_2)$ is isospectral, i.e. the spectrum is given by $\mbox{spec}(L(t,s))=\{\sqrt{3},1/\sqrt{3},0,-1/\sqrt{3},-\sqrt{3}\}$, where each eigenvalue occurs with multiplicity $2$.

\smallskip

One verifies easily that for given $s\in\R$, i.e. a fixed geodesic $c$ in $L_6(p)$, the vectors
\begin{align*}
&(0, 0, -\sin(2s), -\cos(2s), 0, 0, 0, 1, 0, 0)^{tr}, \\
&(0, 0, 0, 0, \sin(s),\cos(s), 0, 0, 0, 0)^{tr}, \\
&(0, 0, 0, 0, \cos(s), -\sin(s), 0, 0, 0, 0)^{tr},\\
&(0, 0, -\cos(2s), \sin(2s), 0, 0, 1, 0, 0, 0)^{tr}
\end{align*}
constitute a basis of $E(c)$. Hence we have $\dim E(c)=4$.

\smallskip

From this we get $E=\mbox{span}\{e_3,e_4,e_5,e_6,e_7,e_8\}$,
where $e_i$ denotes the $i$-th unit vector in $\R^{10}$.
Furthermore, we obtain $L(t,s)E \not\subset E^{\perp}$.
Note that even $L_0E\not\subset E^{\perp}$. 
\end{counterexample*}

The preceding counterexample clearly shows that we can not
deduce the identity $L_0E=E^{\perp}$ as long as we just deal with
the linear isospectral family at one fixed focal submanifold.
In order to generate such an identity (which would hold if isoparametric hypersurfaces with $g=6$ and $m=2$ are indeed homogeneous) one has to analyze the interaction of the isospectral families at different focal submanifolds.

\begin{remark}
In the proof of Proposition 6.2 in \cite{mierr}, Miyaoka considers
the linear isospectral family at one fixed focal submanifold.
Furthermore, she brings the \lq global symmetry\rq \,into play.
However, she uses this identity only to prove that 
a certain subspace $W\subset T_{\overline{p}}M_6\cong\R^{10}$ actually coincides with the orthogonal complement $E^{\perp}$ of $E$ and not to show $L_0E=E^{\perp}$.
\end{remark}

Finding a successful way to analyze the interaction of the isospectral families at different focal submanifolds is one of the central problems that has to be solved in order to classify isoparametric hypersurfaces in spheres with $g=6$ and $m=2$.

\end{document}